\documentclass[11pt,a4paper] {article}

 \usepackage{a4}
\usepackage{amsmath} 
\usepackage{amssymb}
\usepackage{amsthm}
\usepackage{graphicx}
\usepackage{tikz}
\usepackage{subfigure}
\usepackage{float}

\usepackage{color}

\def\+{\oplus}

\newcommand{\R}{{\mathbb R}}

\newcommand{\cG}{{\mathcal G}}

\newcommand{\cJ}{{\mathcal J}}

\newcommand{\cC}{{\mathcal C}}

\newcommand{\cM}{{\mathcal M}}

\newcommand{\cT}{{\mathcal T}}

\newcommand{\cR}{{\mathcal R}}

\renewcommand{\epsilon}{\varepsilon}

\newcommand{\Ga}{{\Gamma}}

\renewcommand{\l}{\lambda}

\newcommand{\FL}{{\rm{FL}}}

\newcommand{\tH}{\widetilde{\rm{H}}}

\newcommand{\ds}{\displaystyle}
\def\squareforqed{\hbox{\rlap{$\sqcap$}$\sqcup$}}
\def\qed{\ifmmode\else\unskip\quad\fi\squareforqed}
\def\smartqed{\def\qed{\ifmmode\squareforqed\else{\unskip\nobreak\hfil
\penalty50\hskip1em\null\nobreak\hfil\squareforqed
\parfillskip=0pt\finalhyphendemerits=0\endgraf}\fi}}

 \newtheorem{theorem}{\textbf{Theorem}}[section]
 \newtheorem{remark}[theorem]{\textbf{Remark}}
 \newtheorem{lemma}[theorem]{\textbf{Lemma}}

 \newtheorem{corollary}[theorem]{\textbf{Corollary}}
 \newtheorem{proposition}[theorem]{\textbf{Proposition}}
 \newtheorem{definition}[theorem]{\textbf{Definition}}
 
\numberwithin{equation}{section}
\title{Effective transmission conditions for Hamilton-Jacobi equations defined on two domains separated by an oscillatory interface} 
\author{Yves Achdou \thanks { Univ. Paris Diderot, Sorbonne Paris Cit{\'e}, Laboratoire Jacques-Louis Lions, UMR 7598, UPMC, CNRS, F-75205 Paris, France.
 achdou@ljll.univ-paris-diderot.fr},
Salom{\'e} Oudet \thanks {IRMAR, Universit{\'e} de Rennes 1, Rennes, France},
Nicoletta Tchou \thanks {IRMAR, Universit{\'e} de Rennes 1, Rennes, France, nicoletta.tchou@univ-rennes1.fr}
}
\begin{document}

\maketitle
\begin{abstract}
  We consider a family of optimal control problems in the plane  with  dynamics and running costs possibly discontinuous across an oscillatory interface $\Gamma_\epsilon$. The oscillations of the interface have small period and amplitude, both of the order of $\epsilon$, and  the  interfaces  $\Gamma_\epsilon$ tend to a straight line $\Gamma$.
  We study the asymptotic behavior as $\epsilon\to 0$. We prove that the value function tends to the solution of  Hamilton-Jacobi equations in the two half-planes limited by  $\Gamma$, with an effective transmission condition on $\Gamma$ keeping track of the oscillations of $\Gamma_\epsilon$.
\end{abstract}

\section{Introduction}
\label{sec:setting}
The goal of this paper is to study the asymptotic behavior as $\epsilon\to 0$ of the value function of an optimal 
control problem in $\R^2$  in which the running cost and dynamics may jump across a periodic oscillatory interface $\Gamma_\epsilon$, when 
the oscillations of $\Gamma_\epsilon$ have  a small amplitude and  period, both of the order of $\epsilon$. 
 The interface $\Gamma_\epsilon$ separates two unbounded regions of $\R^2$, $\Omega_\epsilon^L$ and $\Omega_\epsilon^R$. 
To characterize the optimal control problem,  one has to specify the admissible dynamics at a point $x\in \Gamma_\epsilon$: in our setting, no mixture is allowed at the interface, 
i.e. the admissible dynamics 
are the ones corresponding to the subdomain $ \Omega_\epsilon^L$ {\bf and} entering $ \Omega_\epsilon^L$,
 or corresponding to the subdomain $ \Omega_\epsilon^R$ {\bf and} entering $ \Omega_\epsilon^R$. Hence the situation differs from those studied in the 
articles of G. Barles, A. Briani and E. Chasseigne \cite{barles2011bellman,barles2013bellman} and of  G. Barles, A. Briani, E. Chasseigne and N. Tchou \cite{barles:hal-00986700}, 
in which mixing is allowed at the interface. The optimal control problem under consideration has been first studied in \cite{oudet2014}:  the value function is
characterized as the viscosity solution of a Hamilton-Jacobi equation
 with special transmission conditions on $ \Gamma_\epsilon$; a comparison principle for this problem is proved in \cite{oudet2014} with arguments 
 from the theory of optimal control similar to those introduced in \cite{barles2011bellman,barles2013bellman}. In parallel to \cite{oudet2014}, 
Imbert and Monneau have studied similar problems
from the viewpoint of PDEs, see \cite{imbert:hal-01073954},  and have obtained comparison  results for quasi-convex Hamiltonians.
In particular,  \cite{imbert:hal-01073954} contains a characterization of the viscosity solution of the transmission problem
 with a reduced set of test-functions;  this characterization will be used in the present work.
Note that \cite{oudet2014,imbert:hal-01073954} can be seen as extensions of articles 
devoted to the analysis of Hamilton-Jacobi equations on networks, see \cite{MR3057137,MR3023064,achdou:hal-00847210,imbert:hal-00832545}, 
because the notion of interface used there can be seen as 
a generalization of the notion of vertex (or junction) for a network.
\\
We will see that  as $\epsilon$ tends to $0$, the value function converges  to the solution of an effective problem
related to a flat interface $\Gamma$,  with Hamilton-Jacobi equations in the  half-planes limited by $\Gamma$ and a transmission condition on $\Gamma$.
\\
Whereas the partial differential equation  far from the interface is unchanged, the main difficulty consists in finding the effective transmission condition on $\Gamma$.
Naturally, the latter depends on the dynamics and running cost  but also 
keeps  memory of the vanishing oscillations.  The present work is closely related to two recent articles, \cite{MR3299352} and \cite{galise:hal-01010512},
  about singularly  perturbed problems  leading to effective Hamilton-Jacobi equations on networks. 
Indeed, an effective Hamiltonian corresponding to  trajectories staying close to   the junction  was first
obtained in \cite{MR3299352} as the limit of a sequence of ergodic constants corresponding to larger and larger bounded subdomains.
  This construction was then used in \cite{galise:hal-01010512} in a different case. Let us briefly describe the singular perturbation problems studied in \cite{MR3299352} and \cite{galise:hal-01010512}:  in \cite{MR3299352},  some of the authors of the present paper study a family of star-shaped planar domains  $D^\epsilon$ made of $N$ non intersecting semi-infinite strips  of thickness $\epsilon$ and of a central region whose diameter is proportional to $\epsilon$.  As  $\epsilon \to 0$,  the domains $D^\epsilon$ tend to a  network $\cG$ made of $N$ half-lines sharing an endpoint $O$, named the vertex or junction point.  For infinite horizon optimal control problems  in which the state is constrained to remain in the closure of $D^\epsilon$, the value function tends to the solution of a Hamilton-Jacobi equation on $\cG$, with an effective  transmission condition at $O$. 
  In \cite{galise:hal-01010512},  Galise, Imbert and Monneau study a family of Hamilton-Jacobi equations in a simple network composed of two half-lines with a perturbation of the Hamiltonian localized in a small region close to the junction. 
\\
In the proof of convergence, we will see that the main technical point lies in the construction of correctors  and in their use in the perturbed test-function method of Evans, see \cite{MR1007533}. As in \cite{MR3299352} and \cite{galise:hal-01010512}, an important difficulty  comes from the unboundedness of the domain in which the correctors are defined.  
The strategies for passing to the limit in  \cite{MR3299352} and \cite{galise:hal-01010512} differ: the method proposed in \cite{MR3299352} consists of contructing an infinite family of correctors related to the vertex, while in  \cite{galise:hal-01010512}, only one corrector related to the vertex is needed thanks to the use of the above mentioned  reduced set of test-functions.  Arguably, the strategy proposed in \cite{MR3299352} is more natural and that in \cite{galise:hal-01010512}  is simpler. For this reason,  the technique 
 implemented in the present work for proving the convergence to the effective problem will be closer to the one proposed in \cite{galise:hal-01010512}. Note  that similar techniques are used in the very recent work \cite{forcadel:hal-01097085}, which deals with applications to traffic flows. The question of the correctors in unbounded domains has recently been addressed by P-L. Lions in his lectures at Coll{\`e}ge de France, \cite{PLL}, precisely in january and february 2014: the  lectures dealt with recent and still unpublished results obtained in collaboration with T. Souganidis on the asymptotic behavior of solutions of Hamilton-Jacobi equations in a periodic setting with some localized defects.  Finally, we stress the fact that the technique proposed in the present work is not specific to the transmission condition imposed on $\Gamma_\epsilon$. 

The paper is organized as follows: in the remaining part of \S~\ref{sec:setting}, we set the problem and give the main result. In Section \ref{sec:stra-geom}, we show that the problem is equivalent to a more convenient one,  set in a straightened fixed geometry.  In \S~\ref{sec:asymptotic_behavior}, we study the asymptotic behavior far from the interface and introduce some ingredients that will be useful to define the effective transmission  condition. In \S~\ref{sec:new-hamilt-involv}, we define the effective cost/Hamiltonian for moving along the effective interface, and related correctors. This is of course a key step in the study of the asymptotic behavior. Section~\ref{sec:furth-results-corr} deals with further properties of the correctors, in particular their growth at infinity. The comparison result for the effective problem is stated in \S~\ref{sec:comp-princ-refd}, and the proof of the main convergence theorem is written in \S~\ref{sec:proof-main-result}.

\subsection{The geometry}
\label{sec:geometry}
Let $(e_1, e_2)$ be an orthonormal basis of $\R^2$: $ e_1=\left(\begin{matrix}
1\\0\end{matrix} \right)$, $e_2=\left(\begin{matrix}0\\1\end{matrix} \right)$.
Let $g : \R \to \R$ be a $\mathcal{C}^2$-function, periodic with period $1$. For $\epsilon>0$, let
$(\Omega^L_\epsilon, \Ga_\epsilon, \Omega^R_\epsilon)$ be  the following partition of $\R^2$:
\begin{eqnarray}
  \label{eq: def_Ga_epsilon}
\Ga_\epsilon = \left\lbrace (x_1,x_2)\in \R^2 : x_1=\epsilon g \left(\frac{x_2}{\epsilon}\right) \right\rbrace,\\
\label{eq: def_Omega_1_epsilon}
\Omega^L_\epsilon  =\left\lbrace (x_1,x_2)\in \R^2 : x_1<\epsilon g \left(\frac{x_2}{\epsilon}\right) \right\rbrace,\quad\quad 
\Omega^R_\epsilon = \left\lbrace (x_1,x_2)\in \R^2 : x_1>\epsilon g \left(\frac{x_2}{\epsilon}\right) \right\rbrace.
\end{eqnarray}
Note that  $\partial \Omega^L_\epsilon= \partial \Omega^R_\epsilon= \Ga_\epsilon$. For $x\in \Gamma_\epsilon$, the vector
 \begin{equation}\label{eq: expression_n_epsilon}
n_\epsilon(x)=\left(\begin{matrix}
1\\
-g'(\frac{x_2}{\epsilon})
\end{matrix}\right)
\end{equation}
is  normal to $\Ga_\epsilon$ and oriented from $\Omega^L_\epsilon$ to $\Omega^R_\epsilon$. With
\begin{equation}
\label{eq:def_sigma^i}
\sigma^L= -1,\quad\quad \sigma^R=1,
\end{equation} 
the vector $\sigma^i n_\epsilon(x)$ is normal to $\Ga_\epsilon$ at the point $x\in \Ga_\epsilon$ and points toward $\Omega_\epsilon^i$, for $i=L,R$.\\
The geometry obtained at the limit when $\epsilon\to 0$ can also be found by  taking $g=0$ in the definitions above: let $(\Omega^L, \Ga, \Omega^R)$ be the  partition of $\R^2$ defined by
\begin{eqnarray}
\label{eq: def_Ga}
\Ga  = \left\lbrace (x_1,x_2)\in \R^2 : x_1=0 \right\rbrace,
\\
\label{eq: def_Omega_1}
\Omega^L  =\left\lbrace (x_1,x_2)\in \R^2 : x_1<0) \right\rbrace, \quad
\Omega^R  = \left\lbrace (x_1,x_2)\in \R^2 : x_1>0) \right\rbrace.
\end{eqnarray}
One sees that $\partial \Omega^L= \partial \Omega^R= \Ga$ and that for all $x\in \Gamma$, the unit normal vector to $\Ga$ at $x$ pointing toward $\Omega^R$ is 
 $n(x)=e_1$. The two kinds of geometry are represented in Figure \ref{fig: the_geometry}.
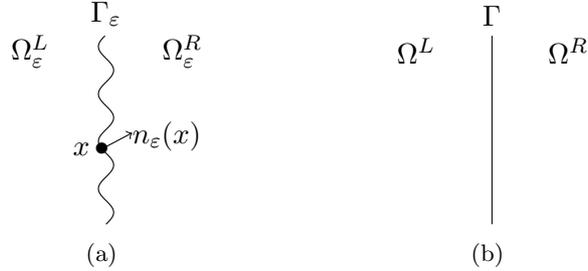
\begin{figure}[H]
\begin{center}
\subfigure[]{
    \label{fig:perturbed_geometry}
\begin{tikzpicture}[scale=0.50]
\draw [domain=0:5, samples=200, smooth]
plot ({(1/5)*sin(5*\x r)},\x) ;
\draw ({(1/5)*sin(5*2 r)},2) node{$\bullet$};
\draw[->] ({(1/5)*sin(5*2 r)},2) -- ++ (28:0.9);
\draw (-0.16,2) node[left]{$x$};
\draw (1.6,3) node[right,below]{$n_\epsilon(x)$};
\draw (0,5) node[above]{$\Ga_\epsilon$};
\draw (-2,3.9) node[above]{$\Omega^L_\epsilon$};
\draw (2,3.9) node[above]{$\Omega^R_\epsilon$};
\end{tikzpicture}
}
\quad \quad \quad \quad \quad
\subfigure[]{
    \label{fig:straight_geometry}
    \begin{tikzpicture}[scale=0.50]
\draw (0,0) -- (0,5);
\draw (0,5) node[above]{$\Ga$};
\draw (-2,4) node[above]{$\Omega^L$};
\draw (2,4) node[above]{$\Omega^R$};
\end{tikzpicture}
}
\caption[Optional caption for list of figures]{\subref{fig:perturbed_geometry}: $\Gamma_\epsilon$ is an oscillating interface with an amplitude and period of $\epsilon$
. \subref{fig:straight_geometry}: the geometry obtained at the limit when $\epsilon\to 0$  }
\label{fig: the_geometry}
\end{center}
\end{figure}

\subsection{The optimal control problem in  $\Omega^L_\epsilon\cup \Omega^R_\epsilon \cup \Ga_\epsilon$}
\label{optimal1}
We consider infinite-horizon optimal control problems which have different dynamics and running costs in the regions $\Omega^i_\epsilon$, $i=L, R$.
  The sets of controls associated to the index $i=L,R$ will be called $A^i$; 
similarly, the notations $f^i$ and  $\ell^i$ will be used for the dynamics and running costs.
The following  assumptions will be made in all the work
\subsubsection{Standing Assumptions}
\label{sec:assumptions}
\begin{description}
\item{[H0]} $A$ is a metric space (one can take $A=\R^m$). For $i=L,R$,  $A^i$ is a non empty compact subset of $A$ and
 $f^i: \R^2\times A^i \to \R^2$ is  a continuous bounded function. The sets $A^i$ are disjoint.
 Moreover, there exists $L_f>0$ such that for any $i=L,R$, $x,y\in \R^2$ and $a\in A^i$,
  \begin{displaymath}
    |f^i(x,a)-f^i(y,a)|\le L_f |x-y|.
  \end{displaymath}
 Define  $M_f=   \max_{i=L, R}   \sup_{x\in \R^2, a \in    A^i } | f^i(x,a)| $.
  The notation $F^i(x)$ will be used for the set $F^i(x)=\{f^i(x,a), a\in A^i\} $.
\item{[H1]} For $i=L,R$, the function $\ell^i: \R^2\times A^i\to \R$ is  continuous and bounded. 
There is a modulus of continuity $\omega_\ell$ 
such that for any $i=L,R$, $x,y\in \R^2$ and $a\in A^i$, 
\begin{displaymath}
|\ell^i(x,a)-\ell^i(y,a)|\le \omega_{\ell} (|x-y|).
\end{displaymath}
 Define  $M_\ell=   \max_{i=L, R}   \sup_{x\in \R^2, a \in    A^i } | \ell^i(x,a)| $.
\item{[H2] }  For any $i=L,R$ and $x\in \R^2$, the non empty  set $\FL^i(x)= \{ (f^i(x,a), \ell^i(x,a) ) , a\in A^i\}$ is  closed and convex. 
\item{[H3] } There is a real number $\delta_0>0$ such that for $i=L,R$ and all $x\in \Ga_\epsilon$, 
$   B(0,\delta_0) \subset F^i(x)$.
\end{description}
We stress the fact that all the  results below hold provided  the latter  assumptions are satisfied,  
although, in order to avoid tedious repetitions, we will not mention them explicitly in the statements.
\\
We refer to \cite{achdou:hal-00847210} and \cite{oudet2014} for comments on the assumptions and the  genericity of the model,
stressing in particular that the sets $A^L, A^R$ can always been supposed disjoint.

\subsubsection{The optimal control problem}
\label{sec:optim-contr-probl}
Let the closed set $\cM_\epsilon$ be defined as follows:
\begin{equation}
  \label{eq: def-M_epsilon}
  \cM_\epsilon=\left\{(x,a);\; x\in \R^2,\quad  a\in A^i  \hbox{ if } x\in \Omega_\epsilon^i,\; i=L,R,\;  \hbox{ and } a     \in A^L\cup A^R  \hbox{ if } x \in \Ga_\epsilon\right \}.
  \end{equation}
  The dynamics $f_\epsilon$ is a function defined in  $\cM_\epsilon$ with values in $\R^2$: 
\begin{displaymath}
\forall (x,a)\in \cM_\epsilon,\quad\quad   f_\epsilon(x, a)=\left\{
    \begin{array}[c]{ll}
      f^i(x,a) \quad &\hbox{ if } x\in \Omega_\epsilon^i, \\
      f^i(x,a) \quad &\hbox{ if } x\in\Ga_\epsilon \hbox{ and }     a\in A^i.
    \end{array}
\right.
\end{displaymath}
The function $f_\epsilon$ is continuous on $\cM_\epsilon$ because the sets $A^i$ are disjoint. Similarly, let the running cost $\ell_\epsilon: \cM_\epsilon\to \R$ be given by
\begin{displaymath}
\forall (x,a)\in \cM_\epsilon,\quad\quad   \ell_\epsilon(x, a)=\left\{
    \begin{array}[c]{ll}
      \ell^i(x,a) \quad &\hbox{ if } x\in \Omega_\epsilon^i, \\
      \ell^i(x,a) \quad &\hbox{ if } x\in\Ga_\epsilon \hbox{ and }     a\in A^i.
    \end{array}
\right.
\end{displaymath}
For $x\in \R^2$, the set of admissible  trajectories starting from $x$ is 
\begin{equation}
  \label{eq:2}
\cT_{x,\epsilon}=\left\{
  \begin{array}[c]{ll}
    ( y_x, a)  \ds \in L_{\rm{loc}}^\infty( \R^+; \cM_\epsilon): \quad   & y_x\in {\rm{Lip}}(\R^+; \R^2), 
 \\  &\ds  y_x(t)=x+\int_0^t f_\epsilon( y_x(s), a(s)) ds \quad  \forall t\in \R^+ 
  \end{array}\right\}.
\end{equation}
The cost associated to the trajectory $ ( y_x, a)\in \cT_{x,\epsilon}$ is 
\begin{equation}
  \label{eq:46}
  \cJ_\epsilon(x;( y_x, a) )=\int_0^\infty \ell_\epsilon(y_x(t),a(t)) e^{-\lambda t} dt,
\end{equation}
with $\lambda>0$. The value function of the infinite horizon optimal control problem is 
\begin{equation}
  \label{eq:4}
v_\epsilon(x)= \inf_{( y_x, a)\in \cT_{x,\epsilon}}   \cJ_\epsilon(x;( y_x, a) ).
\end{equation}
\begin{proposition}\label{sec:assumption}
 The value function $v_\epsilon$ is bounded and continuous in $ \R^2$.
\end{proposition}
\begin{proof}
This result is classical and can be proved with the same arguments as in \cite{MR1484411}.
\end{proof}

\subsection{The Hamilton-Jacobi equation}
\label{sec:hamilt-jacobi-equat}
Similar optimal control problems have recently been studied in \cite{achdou:hal-00847210,imbert:hal-00832545,oudet2014,imbert:hal-01073954}.
It turns out that $v_\epsilon$ can be characterized as the viscosity solution of a Hamilton-Jacobi equation with a discontinuous Hamiltonian,
(once the notion of viscosity solution has been specially tailored to cope with the above mentioned discontinuity).  We briefly recall the definitions used e.g. in \cite{oudet2014}.
\subsubsection{Test-functions}\label{sec:test-functions}
\begin{definition}
\label{adtest}
For $\epsilon >0$,  the function $\phi: \R^2\to \R$ is an admissible ($\epsilon$)-test-function if
$\phi$ is continuous in $\R^2$  and for any $i\in \{L,R\}$, $\phi|_{\overline{\Omega^i_\epsilon}} \in\cC^1(\overline{\Omega^i_\epsilon})$.
\\The set of admissible test-functions is noted $\cR_\epsilon$. If $\phi \in \cR_\epsilon$, $x\in \Ga_\epsilon$ and $i\in \{L,R\}$, we  set
$D\phi^i(x)= \lim_{\overset{ x'\to x}{x'\in \Omega^i_\epsilon}}D\phi(x')$.
\end{definition}
\subsubsection{Hamiltonians}\label{sec:hamiltonians}
For $i=L,R$, let the Hamiltonians $H^i:\R^2\times \R^2\rightarrow \R $  and  $H_{ \Ga_\epsilon}:\Ga_\epsilon\times \R^2\times \R^2\to \R$  be defined by
\begin{eqnarray}
  \label{eq:7}
H^i(x,p)&=& \max_{a\in A^i} (-p \cdot f^i(x,a) -\ell^i(x,a)),\\
  \label{eq:8}
H_{ \Ga_\epsilon} (x,p^L,p^R)&=& \max \{ \;H_{ \Ga_\epsilon}^{+,L}(x,p^L),H_{ \Ga_\epsilon}^{+,R} (x,p^R)\},
\end{eqnarray}
where, with   $n_\epsilon(x)$ and $\sigma^i$ defined in  \S~\ref{sec:geometry},
\begin{equation}\label{eq:1}
H_{ \Ga_\epsilon}^{+,i} (x,p)= \max_{a\in A^i \hbox{ s.t. }   \sigma^if^i(x,a)\cdot n_\epsilon(x)\ge 0} (-p\cdot f^i(x,a) -\ell^i(x,a)), \quad \forall x\in \Ga_\epsilon, \forall p\in \R^2.
\end{equation}

\subsubsection{Definition of viscosity solutions}\label{sec:definitiona}
We now recall the definition of  a viscosity solution of
\begin{equation}\label{HJaepsilon}
    \lambda u+{\cal{H}}_\epsilon(x, Du)=0.
\end{equation}
\begin{definition}\label{netviscoa}
\begin{itemize}
\item An upper semi-continuous function $u:\R^2\to\R$ is a subsolution of \eqref{HJaepsilon}
 if for any $x\in \R^2$, any $\phi\in\cR_\epsilon$ s.t. $u-\phi$ has a local maximum point at $x$, then
 \begin{eqnarray}
  \label{eq:5bis}
\l u(x)+H^i(x,D\phi^i(x))&\le 0, \quad \quad &    \hbox{ if }x\in \Omega^i_\epsilon,\\
\label{eq:5bisgamma}
\l u(x)+H_{ \Ga_\epsilon} (x,D\phi^L(x),D\phi^R(x))&\le 0, \quad \quad &    \hbox{ if }x\in\Ga_\epsilon,
 \end{eqnarray}
see Definition \S~\ref{sec:geometry} for the meaning of  $D\phi^i(x)$ if $x\in \Gamma_\epsilon$.

\item  A lower semi-continuous function $u:\R^2\to\R$ is a supersolution of \eqref{HJaepsilon}
 if for any $x\in \R^2$, any $\phi\in\cR_\epsilon$ s.t. $u-\phi$ has a local minimum point at $x$, then
 \begin{eqnarray}
  \label{eq:6bis} \l u(x)+H^i(x,D\phi^i(x))&\geq 0,  \quad \quad &    \hbox{ if }x\in \Omega^i_\epsilon,\\
\label{eq:6bisgamma} \l u(x)+H_{ \Ga_\epsilon} (x,D\phi^L(x),D\phi^R(x))&\ge 0 \quad \quad &    \hbox{ if }x\in\Ga_\epsilon.
 \end{eqnarray}
  \item A continuous function $u:\R^2\to\R$ is a   viscosity solution of \eqref{HJaepsilon}   if it is both a viscosity sub and supersolution   of \eqref{HJaepsilon}.
\end{itemize}
\end{definition}

\subsubsection{Characterization of $v_\epsilon$ as a viscosity solution of (\ref{HJaepsilon})}
\label{sec:char-v-as}
The following theorem will be proved below, see Theorem~\ref{existence-epsilon}, by finding an equivalent optimal control problem in a straightened fixed geometry
and using some results contained in \cite{oudet2014}:
\begin{theorem}
  \label{existence-epsilon_a}
The value function $v_\epsilon$  defined in \eqref{eq:4}  is the unique bounded  viscosity solution of \eqref{HJaepsilon}.
\end{theorem}

\subsection{Main result and organization of the paper}
\label{sec:main-result-organ}We now state our main result:   
\begin{theorem}\label{th:convergence_result}
   As $\epsilon\to 0$, $v_\epsilon$ converges uniformly  to $v$ the unique bounded viscosity solution of 
  \begin{eqnarray}
    \label{def:HJeffective1}
    \lambda v(z)+ H^i(z, D v(z))  = 0 & &\hbox{if } z\in \Omega^i,\\
    \label{def:HJeffective2}
    \lambda v(z)+\max\left(E(z_2,\partial_{z_2}v(z)), H_\Ga(z, D v^L(z),D v^R(z)) \right) = 0  & &\hbox{if } z=(0,z_2)\in \Ga,
  \end{eqnarray}
  which we note for short
\begin{equation}\label{def:HJeffective_short}
  \lambda v+{\cal{H}}(x, Dv)=0.
\end{equation}
The Hamiltonians $H^i$, $H_\Gamma$  and $E$ are respectively defined in \eqref{eq:7},  \eqref{def:H_gamma} below, and \eqref{eq:def_E} below. 
\end{theorem}
Let us list the  notions which are needed by Theorem~\ref{th:convergence_result}  and give a few comments:
\begin{enumerate}
\item Problem~(\ref{def:HJeffective_short}) is a transmission problem across the interface $\Gamma$, with the effective transmission condition~(\ref{def:HJeffective2}). The notion of viscosity solutions of (\ref{def:HJeffective_short}) 
is similar to the one proposed in  Definition~\ref{netviscoa}, replacing $\Gamma_\epsilon$ with $\Gamma$. 
\item Note that the Hamilton-Jacobi equations in $\Omega^L$ and $\Omega^R$ are  directly inherited from (\ref{eq:5bis}): this is quite natural, 
since the interface $\Gamma_\epsilon$ oscillates with an  amplitude  of the order of $\epsilon$, which therefore vanishes as $\epsilon\to 0$.
\item  The Hamiltonian  $H_\Gamma$   appearing in the effective transmission condition at the junction is defined in \S~\ref{sec:effect-hamilt-gamma}, precisely in (\ref{def:H_gamma}); it is built by considering only the dynamics related to $ \Omega^i$ which point from $\Gamma$ toward $\Omega^i$, for $i=L,R$.
\item The effective Hamiltonian $E$ is the only ingredient in the effective problem that  keeps track of the oscillations of $\Gamma_\epsilon$, i.e. of the function $g$. It is constructed in \S~\ref{sec:new-hamilt-involv}, see (\ref{eq:def_E}),  as 
the limit of a sequence of ergodic constants related to larger and larger bounded subdomains. This is reminiscent of a construction first performed in \cite{MR3299352} for 
 singularly perturbed problems in optimal control leading to Hamilton-Jacobi equations posed on a network. A similar construction can also be found in \cite{galise:hal-01010512}.
\item For proving  Theorem~\ref{th:convergence_result}, 
 the chosen strategy is reminiscent of  \cite{galise:hal-01010512},
 because it relies on the construction of a single corrector,
 whereas the method proposed in  \cite{MR3299352} requires the construction of an infinite family of correctors.  This will be done in \S~\ref{sec:new-hamilt-involv} and the slopes at infinity of the correctors will be studied in \S~\ref{sec:furth-results-corr}. 
\end{enumerate}

\section{Straightening  the geometry}
\label{sec:stra-geom}
It will be convenient to use a change of variables depending on $\epsilon$ and set the problem in a straightened and fixed geometry.
\subsection{A change of variables}\label{sec:changeofvar}
The following change of variables can be used to write the optimal control problem in a fixed geometry: for $x\in \R ^2$, take $z=G(x)=  \left(
  \begin{array}[c]{c}
    x_1-\epsilon g(\frac{x_2}{\epsilon})\\
x_2
  \end{array}
\right)$. We see that  $G^{-1}(x)=  \left(
  \begin{array}[c]{c}
    x_1+\epsilon g(\frac{x_2}{\epsilon})\\
x_2
  \end{array}
\right)$. The oscillatory interface $\Gamma_\epsilon$ is mapped onto $\Gamma=\{z: z_1=0\}$ by $G$.
 The  Jacobian of $G$ is 
\begin{equation}
\label{jacobian}
J_\epsilon(x)=\left(\begin{matrix}
1 &-g'(\frac{x_2}{\epsilon})\\
0 &1
\end{matrix}\right), 
\end{equation}
and it inverse is 
$
J_\epsilon^{-1}(x)=\left(\begin{matrix}
1 & g'\left(\frac{x_2}{\epsilon} \right)\\
0 & 1
\end{matrix} \right)$.
The following properties will be useful: for any $x\in \R^2$,
\begin{eqnarray}
\label{prop:matrix_norm_J_epsilon_2} J_\epsilon(G^{-1}(x))= J_\epsilon(x) \quad&\hbox{and }&\quad  \displaystyle  J_\epsilon^{-1}(G^{-1}(x))=J^{-1}_\epsilon(x),
\\
\label{prop:matrix_norm_J_epsilon_3}
\sup_{X\in \R^2, |X|\le 1}|J_\epsilon(x)   X| \le \sqrt 2(1+\parallel g' \parallel_\infty)\quad &\hbox{and}&\quad 
\!\!\!\!\!\!\!\! \displaystyle \sup_{X\in \R^2, |X|\le 1}|J_\epsilon^{-1}(x)  X| \le \sqrt 2 (1+\parallel g' \parallel_\infty),
\end{eqnarray}
where $|\cdot|$ stands for the  euclidean norm. Note that (\ref{prop:matrix_norm_J_epsilon_2}) holds because $G$ and $G^{-1}$ leave $x_2$ unchanged, and $J_\epsilon$ only depends on $x_2$.

\subsection{The optimal control problem in the straightened geometry}
\label{subsec:straight_optima_control}
For $i=L,R$, we define the new dynamics $\tilde f_\epsilon^i$ and running costs $\tilde \ell_\epsilon^i$ as 
\begin{eqnarray}
& & \begin{array}{lllllll}
\tilde f_\epsilon^i :& \bar \Omega^i\times A^i &\to & \R^2,\quad \quad
& (z,a) & \mapsto &J_\epsilon(z) f^i\left(G^{-1}(z),a\right),
\end{array}\\
& & \begin{array}{lllllll}
\tilde \ell_\epsilon^i :& \bar \Omega^i\times A^i &\to & \R,\quad \quad
& (z,a) & \mapsto & \ell^i\left(G^{-1}(z),a\right).
\end{array}
\end{eqnarray}
We deduce  the following properties from the standing assumptions $[\rm{H}0]$-$[\rm{H}3]$:
 \begin{description}
\item{$[\tH0]_\epsilon$} For $i=L,R$ and $\epsilon>0$, the function $\tilde f_\epsilon^i$ is continuous and bounded. 
Moreover, there exists  $\tilde L_f(\epsilon)>0 $ and $\tilde M_f>0$ such that 
for any $z,z'\in \bar \Omega^i$ and $a\in A^i$,
\begin{eqnarray*}
&    |\tilde f_\epsilon^i(z,a)-\tilde f_\epsilon^i(z',a)|\le \tilde L_f(\epsilon) |z-z'|,\\
& | \tilde f_\epsilon^i(z,a)| \le \tilde M_f.
\end{eqnarray*}
\item{$[\tH1]_\epsilon$}  For $i=L,R$ and $\epsilon>0$, the function $\tilde \ell_\epsilon^i$ is continuous and bounded. Moreover, 
if we set $\tilde \omega_{\ell}(t)=\omega_{\ell}(\sqrt 2(1+\parallel g' \parallel_\infty)t)$, then for any $z,z'\in \bar \Omega^i$ and $a\in A^i$,
\begin{eqnarray*}
&|\tilde \ell_\epsilon^i(z,a)-\tilde \ell_\epsilon^i(z',a)|\le \tilde \omega_{\ell} (|z-z'|),\\  
&| \tilde \ell_\epsilon^i(x,a)| \le M_\ell,
\end{eqnarray*}
  the constants $M_\ell$ and the modulus of continuity $\omega_\ell(\cdot)$ being introduced in $[\rm{H}1]$.
\item{$[\tH2]_\epsilon$}  For any $i=L,R$, $\epsilon>0$ and $x\in \bar \Omega^i$, the non empty set  $\tilde \FL_\epsilon^i(x)= \{ (\tilde f_\epsilon^i(x,a), \tilde \ell_\epsilon^i(x,a) ) , a\in A^i\}$ is closed and convex. 
\item{$[\tH3]_\epsilon$}  For any $i=L,R$ and $\epsilon>0$, if we set $\tilde \delta_0= \frac{\delta_0}{\sqrt 2(1+\parallel g'\parallel_\infty)}$, then for any  $z\in \Ga$, 
$    B(0,\tilde \delta_0) \subset \tilde F_\epsilon^i(z)=\{\tilde f_\epsilon^i(z,a), a\in A^i\}$.
\end{description}
Properties $[\tH0]_\epsilon$ and $[\tH1]_\epsilon$ result from direct calculations.
Property $[\tH2]_\epsilon$ comes from  the fact that linear maps preserve the convexity property.
 In order to prove $[\tH3]_\epsilon$, take $i=L,R$, $z=(0,z_2)\in \Ga$ and $p\in B(0,\tilde \delta_0)$.
 We look for $a\in A^i$ such that $\tilde f_\epsilon^i(z,a)=p$. Using \eqref{prop:matrix_norm_J_epsilon_3}, we see that 
$
| J_\epsilon^{-1}(z) p|\le  \sqrt 2(1+\parallel g' \parallel_\infty) \tilde \delta_0\le \delta_0
$.
Thus from $[\rm{H}3]$, since $   G^{-1}(z) \in \Ga_\epsilon$, there exists  $\bar a\in A^i$ such that 
$
f^i\left( G^{-1}(z) ,\bar a\right) =J_\epsilon^{-1}(z)p$, and we obtain  that $\tilde f_\epsilon^i(z,\bar a)=p$.
\\
Let us now define the counterparts of $\cM_\epsilon$, $ f_\epsilon$ and $\ell_\epsilon$:
\begin{eqnarray}
  \label{eq: def-M}
  \cM=\left\{(x,a);\; x\in \R^2,\quad  a\in A^i  \hbox{ if } x\in \Omega^i,  \hbox{ and } a     \in A^L\cup A^R  \hbox{ if } x \in \Ga\right \},\\
\label{eq:3}
  \forall (z,a)\in \cM,\quad\quad  \tilde f_\epsilon(z, a)=\left\{
    \begin{array}[c]{ll}
      \tilde f_\epsilon^i(z,a) \quad &\hbox{ if } x\in \Omega^i, \\
      \tilde f_\epsilon^i(z,a) \quad &\hbox{ if } x\in\Ga \hbox{ and }     a\in A^i,
    \end{array}
\right.\\
\forall (z,a)\in \cM,\quad\quad  \tilde \ell_\epsilon(z, a)=\left\{
    \begin{array}[c]{ll}
      \tilde \ell_\epsilon^i(z,a) \quad &\hbox{ if } x\in \Omega^i, \\
      \tilde \ell_\epsilon^i(z,a) \quad &\hbox{ if } x\in\Ga \hbox{ and }     a\in A^i.
    \end{array}
\right.
\end{eqnarray}
For $x\in \R^2$, the set of admissible  trajectories starting from $x$ is 
\begin{equation}
  \label{eq:2tilde}
\widetilde \cT_{x,\epsilon}=\left\{
  \begin{array}[c]{ll}
    ( y_x, a)  \ds \in L_{\rm{loc}}^\infty( \R^+; \cM): \quad   & y_x\in {\rm{Lip}}(\R^+; \R^2), 
 \\  &\ds  y_x(t)=x+\int_0^t \tilde f_\epsilon( y_x(s), a(s)) ds \quad  \forall t\in \R^+ 
  \end{array}\right\}.
\end{equation}
Note that $\forall z\in \R^2$,
$(y_z,a)\in  \widetilde\cT_{z,\epsilon}  \Leftrightarrow \left( G^{-1}(y_z(\cdot)),a\right)\in \cT_{G^{-1}(z),\epsilon}$.
The new optimal control problem consists in finding
\begin{equation}
  \label{eq:4_tilde}
\tilde v_\epsilon(z)= \inf_{( y_z, a)\in \widetilde \cT_{z,\epsilon}}   \int_0^\infty \tilde \ell_\epsilon(y_z(t),a(t)) e^{-\lambda t} dt.
\end{equation}

\begin{remark}[Relationship between $v_\epsilon$ and $\tilde v_\epsilon$]
\label{prop:relation_v_v_tilde} 
 For any $z\in \R^2$,
\begin{displaymath}
\tilde v_\epsilon(z)=v_\epsilon(G^{-1}(z))=v_\epsilon\left(z_1+\epsilon g\left(\frac{z_2}{\epsilon} \right),z_2\right).
\end{displaymath}
\end{remark}
\subsection{The Hamilton-Jacobi equation in the straightened geometry}\label{sec:hamilt-jacobi-equat-1}
\subsubsection{Hamiltonians}
\label{sec:hamiltonians_in_new_variables}
If $i\in \{L,R\}$,  the Hamiltonians $\tilde{H}^i_\epsilon:\R^2\times \R^2\rightarrow \R $ are defined by
\begin{equation}
  \label{eq:7new}
\tilde{H}^i_\epsilon(z,p)=    \max_{a\in A^i} \left(- \tilde f^i_\epsilon(z,a)\cdot  p  -\tilde \ell^i_\epsilon(z,a)\right)      = \max_{a\in A^i} \left(- J_\epsilon(z)f^i(G^{-1}(z),a)\cdot  p  -\ell^i(G^{-1}(z),a)\right).
\end{equation} 
 More explicitly,
 \begin{equation*}
 \tilde{H}^i_\epsilon(z,p)= \max_{a\in A^i} \left(
 -\left(\begin{matrix}
 1 & -g'(\frac{z_2}{\epsilon}) \\
 0 & 1
 \end{matrix}\right)f^i((z_1+\epsilon g(\frac{z_2}{\epsilon}), z_2),a)\cdot
 p  -\ell^i((z_1+\epsilon g(\frac{z_2}{\epsilon}), z_2),a)\right).
 \end{equation*} 
If $z\in \Ga$, the Hamiltonian $\tilde{H}_{ \Ga,\epsilon{}}:\Ga\times \R^2\times \R^2\to \R$ is defined by 
\begin{equation}
  \label{eq:8new}
\tilde{H}_{ \Ga, \epsilon} (z,p^L,p^R)=  \max \left( \;\tilde{H}_{ \Ga,\epsilon}^{+,L}(z,p^L),\tilde{H}_{ \Ga,\epsilon}^{+,R} (z,p^R)\right) ,
\end{equation} 
where for $i=1,2$,       $z\in \R^2$, $p^i\in \R^2$, and $\sigma^i$ is defined in \S~\ref{sec:geometry},
\begin{equation*}
\tilde{H}_{ \Ga,\epsilon}^{+,i} (z,p)=
\max_{\begin{array}{c}
a\in A^i \hbox{ s.t. } \\
\sigma^i  \tilde f^i_\epsilon (z,a)\cdot e_1\ge 0
\end{array}
}  (-\tilde f^i_\epsilon(z,a)\cdot  p -\tilde \ell^i_\epsilon(z,a)).
\end{equation*} 
If $z\in \Gamma$, then
 $n_\epsilon(G^{-1}(z))=J_\epsilon(z)^T e_1$, by \eqref{prop:matrix_norm_J_epsilon_2}. Hence,  $\tilde{H}_{ \Ga,\epsilon}^{+,i}$ is the counterpart of  
${H}_{ \Ga,\epsilon}^{+,i}$.

\subsubsection{Definition of viscosity solutions in  the straightened geometry}\label{sec:defin-visc-solut}
\begin{definition}
\label{def:set_test_function_with_straight_geometry}
The function $\phi : \R^2 \to \R$ is an admissible test-function for the fixed geometry if
$\phi$ is continuous in $\R^2$ and for any $i=L,R$, $\phi|_{\bar \Omega ^i} \in \cC^1(\bar \Omega^i)$.
\\
The set of the admissible test-functions is denoted $\cR$. If $\phi \in \cR$, $x\in \Ga$ and $i\in \{L,R\}$, we set
$D\phi^i(x)= \lim_{\overset{ x'\to x}{x'\in \Omega^i}}D\phi(x')$.
Of course,  the partial derivatives of $\phi|_{\bar \Omega^L}$ and $\phi|_{\bar \Omega^R}$ with respect to $x_2$ coincide on $\Gamma$.
\end{definition}
\noindent
We then define the sub/super-solutions and solutions of
\begin{equation}\label{HJaepsilontilde}
    \lambda u+{\tilde{\cal{H}}}_\epsilon(z, Du)=0
\end{equation}
as in Definition \ref{netviscoa}, using the set of test-functions $\cR$, the Hamiltonians $\tilde{H}^i_\epsilon(z,p)$ if $z\in\Omega^i$ and $\tilde{H}_{ \Ga,\epsilon}(z,p^L,p^R)$ if $z\in \Ga$. 

\begin{remark}
\label{th:equivalence_between_v_vtilde_2}
Let $u :\R^2 \to \R$ be an upper semi-continuous (resp. lower semi-continuous) function and $\tilde u :\R^2 \to \R$ be defined by $\tilde u (z) =u(G^{-1}(z))$. Then $u$ is a subsolution (resp. supersolution) of \eqref{HJaepsilon} if and only if $\tilde u$ is a subsolution (resp. supersolution) of \eqref{HJaepsilontilde}.
\end{remark}

\subsubsection{Existence and uniqueness}
\label{sec:existence-uniqueness}
We have seen in Remark \ref{prop:relation_v_v_tilde} that the optimal control problems (\ref{eq:4}) and (\ref{eq:4_tilde}) are equivalent; similarly Remark \ref{th:equivalence_between_v_vtilde_2} tells us that the notions of viscosity solutions of (\ref{HJaepsilon}) and (\ref{HJaepsilontilde}) are equivalent. Therefore, it is enough to focus on (\ref{eq:4_tilde})  and (\ref{HJaepsilontilde}).

\begin{lemma}
   \label{sec:prop-visc-sub-3}
  There exists $r>0$ such that any bounded  viscosity subsolution $u$ of \eqref{HJaepsilon} (resp. \eqref{HJaepsilontilde})  is Lipschitz continuous in $B(\Ga_\epsilon, r)$ (resp. $B(\Ga, r)$) with Lipschitz constant $L_u\le \frac{\lambda \parallel u \parallel_\infty+M_\ell}{\delta_0}$ (resp. $L_u\le \sqrt 2\frac{(\lambda \parallel u \parallel_\infty+M_\ell)(1+\parallel g'\parallel_\infty)}{\delta_0}$), where for $X$ a closed subset of $\R^2$, $B(X, r)$ denotes the set $\{y\in \R^2: dist(y, X)<r\}$.
 \end{lemma}
 \begin{proof}
For  a subsolution $u$ of \eqref{HJaepsilontilde},  the result is exactly \cite[Lemma~2.6]{oudet2014}.\\
If $u$ is a subsolution of \eqref{HJaepsilon}, then $\tilde u (z) =u(G^{-1}(z))$ is a subsolution of \eqref{HJaepsilontilde}, 
 and is therefore Lipschitz continuous in a neighborhood of $\Ga$. Since $u =\tilde u  \circ G $,  $u$ is Lipschitz continuous in a neighborhood of $\Ga_\epsilon$.
\end{proof}

\begin{theorem}[Local comparison principle]
\label{th: local comparison_ps} 
  Let $u$ be a bounded viscosity subsolution of \eqref{HJaepsilon} (resp. \eqref{HJaepsilontilde}),
 and $v$ be a  bounded viscosity supersolution of \eqref{HJaepsilon} (resp. \eqref{HJaepsilontilde}). For any $z\in \R^2$, there exists $r>0$ such that
\begin{equation}
\label{eq: lem local comparison}
\parallel (u-v)_+ \parallel_{L^\infty(B(z,r))} \le \parallel (u-v)_+ \parallel_{L^\infty(\partial B(z,r))}.
\end{equation}
\end{theorem}
\begin{proof}
Let us focus on \eqref{HJaepsilontilde}. If $z\in \Omega^i$,  then we can choose $r>0$ small enough so that $B(z,r)\subset \Omega^i$ and the result is classical.
 If $z\in \Ga$,  the result stems from a direct application of \cite[Theorem~3.3]{oudet2014}. Indeed,
all the  assumptions required by  \cite[Theorem~3.3]{oudet2014} are satisfied thanks to  the properties 
 $[\widetilde{\rm{H}}0]_\epsilon$-$[\widetilde {\rm{H}}3]_\epsilon$. The result for \eqref{HJaepsilon} can be deduced from the latter thanks to Remark \ref{th:equivalence_between_v_vtilde_2}.
\end{proof}

\begin{theorem}[Global comparison principle]
\label{sec:comparison-principle-1}
 Let $u$ be a bounded viscosity subsolution of \eqref{HJaepsilon} (resp. \eqref{HJaepsilontilde}),
 and $v$ be a  bounded viscosity supersolution of \eqref{HJaepsilon} (resp. \eqref{HJaepsilontilde}). 
Then $u\le v$.
\end{theorem}
\begin{proof}
The result for equation \eqref{HJaepsilontilde} stems from a direct application of \cite[Theorem~3.4]{oudet2014}. 
Then we deduce the result for  equation \eqref{HJaepsilon} thanks to Remark \ref{th:equivalence_between_v_vtilde_2}.
\end{proof}

\begin{theorem}
  \label{existence-epsilon}
The value function $v_\epsilon$ (resp. $\tilde v_\epsilon$) defined in \eqref{eq:4} (resp. \eqref{eq:4_tilde}) is the unique bounded  viscosity solution of \eqref{HJaepsilon} (resp. \eqref{HJaepsilontilde}).
\end{theorem}
\begin{proof}
Uniqueness is a direct consequence of  Theorem \ref{sec:comparison-principle-1} for both equations \eqref{HJaepsilon} and \eqref{HJaepsilontilde}. Existence  for  equation \eqref{HJaepsilontilde} is proved in the same way as in \cite[Theorem~2.3]{oudet2014}. Then, existence for \eqref{HJaepsilon} is deduced  from Remark \ref{th:equivalence_between_v_vtilde_2}.
\end{proof}
\begin{remark}
Under suitable assumptions, see \S 4 in \cite{oudet2014},
all the above results hold if we modify \eqref{HJaepsilon} (resp.\eqref{HJaepsilontilde}) by 
adding to  $H_{\Ga_\epsilon}$ (resp. $\tilde H_{\Ga,\epsilon}$)  a  Hamiltonian $H^0_\epsilon$  (resp. $\tilde H^0_\epsilon$) 
correponding to trajectories staying on the junctions.
\end{remark}

\section{Asymptotic behavior in $\Omega^L$ and $\Omega^R$}
\label{sec:asymptotic_behavior}
Our goal is to understand the asymptotic behavior of the sequence $\left(v_\epsilon\right)_\epsilon$ as $\epsilon$ tends to $0$. 
In this section, we are going to see that the Hamilton-Jacobi equations remain unchanged in $\Omega^L$ and $\Omega^R$; this is 
not surprising  because the amplitude of the oscillations of the interface vanishes as $\epsilon\to 0$. Then, we are going to introduce some of the ingredients of the effective boundary conditions on $\Gamma$.
\\
From Remark \ref{prop:relation_v_v_tilde}, the sequence $\left(v_\epsilon\right)_\epsilon$ converges if and only if the sequence $\left(\tilde v_\epsilon\right)_\epsilon$ converges. Moreover, if they converge, the two sequences have the same limit. It will be convenient to focus on  the asymptotic behavior of the sequence $\left(\tilde v_\epsilon\right)_\epsilon$, since the geometry is fixed. It is now classical to consider the relaxed semi-limits
 \begin{equation}
 \label{def:v_tilde_overline_underline}
 \overline{\tilde{v}}(z)={\limsup_\epsilon}^{*} \tilde{v}_\epsilon(z)=\limsup_{z'\to z, \epsilon\to 0}\tilde{v}_\epsilon(z') \quad \mbox{ and } \quad \underline{\tilde{v}}(z)=\underset{\epsilon}{{\liminf}_{*}}\tilde{v}_\epsilon(z)=\liminf_{z'\to z, \epsilon\to 0}\tilde{v}_\epsilon(z').
 \end{equation}
Note that  $ \overline{\tilde{v}}$ and $ \underline{\tilde{v}}$ are well defined, since  $\left(\tilde v_\epsilon\right)_\epsilon$ is uniformly bounded by $\frac{M_\ell}{\lambda}$, see \eqref{eq:4_tilde}.
\subsection{For the Hamilton-Jacobi equations in   $\Omega^L$ and $\Omega^R$, nothing changes}
\label{sec:hamilt-jacobi-equat-2}
\begin{proposition}
  \label{supersub}
 For $i=L,R$, the functions $\overline{\tilde{v}}(z)$ and $\underline{\tilde{v}}(z)$ are respectively a bounded subsolution and a bounded supersolution in $\Omega^i$ of
\begin{equation}
  \label{eq:5bisnew}
\l u(z)+H^i(z,Du(z))= 0,
\end{equation}
where the Hamiltonian $H^i$ is given by \eqref{eq:7}.
\end{proposition}
\begin{proof}
The proof is classical and relies on perturbed test-functions techniques, see\cite{MR1007533}. 
 For a test- function $\phi$  (near a point $\bar z$ for example), the main idea is to construct the perturbed test-function $\phi_\epsilon(z)=\phi(z)+\epsilon\partial_{z_1}\phi(\bar{z})g(\frac{z_2}\epsilon)-\delta$, for a suitable positive number $\delta$.
\end{proof}
\subsection{An ingredient in the effective transmission condition on $\Gamma$: the Hamiltonian $H_\Ga$  inherited from the half-planes}
\label{sec:effect-hamilt-gamma}
For $i\in \{L,R\}$, let us define the Hamiltonian $H^{+,i}$ and $H^{-,i}$:  $\R^2\times \R^2\rightarrow \R $  by
\begin{equation}
  \label{def:H_Gamma_+_i}
H^{\pm,i}(z,p)= \max_{a\in A^i \hbox{ s.t. } \pm \sigma^i f^i(z,a).e_1 \geq 0} (-p. f^i(z,a) -\ell^i(z,a)).
\end{equation}
and   $H_\Ga:\Ga \times \R^2 \times \R^2 \rightarrow \R $  by
\begin{equation}
  \label{def:H_gamma}
H_\Ga(z,p^L,p^R)= \max\left(H^{+,L}(z,p^L),H^{+,R}(z,p^R) \right).
\end{equation}
As in \cite{MR3299352,galise:hal-01010512}, we introduce the functions $ E_0^i: \R\times \R \to \R$, $i=L,R$ and $ E_0: \R\times \R \to \R$:
\begin{eqnarray}
 \label{def:E_0^i}
 E_0^i(z_2,p_2)&=&\min\left\lbrace H^i((0, z_2),p_2e_2+qe_1), \quad q\in \R \right\rbrace,
 \\
  \label{def:E_0}
 E_0(z_2,p_2)&=&\max\left\lbrace  E_0^L(z_2,p_2), E_0^R(z_2,p_2)\right\rbrace.
\end{eqnarray}
The following lemma, which is the same as  \cite[Lemma~2.1]{oudet2014},    deals with some monotonicity properties of   $H^{\pm,i}$: 
\begin{lemma}\label{sec:effect-hamilt-gamma-1}
$\;$
  \begin{enumerate}
  \item For any $(0,z_2)\in \Ga$, $p\in \R^2$,
  $p_1\mapsto H^{+,L}((0,z_2),p+p_1e_1)$   and $p_1\mapsto H^{-,R}((0,z_2),p+p_1e_1)$   are nondecreasing;   $p_1\mapsto H^{-,L}((0,z_2),p+p_1e_1)$  and $p_1\mapsto H^{+,R}((0,z_2),p+p_1e_1)$ are  nonincreasing.
\item For $z_2, p_2\in \R$,  there exist two  unique real numbers $p^{-,L}_0(z_2,p_2)\le  p^{+,L}_0(z_2,p_2) $ such that
  \begin{eqnarray*}
H^{-,L}((0,z_2),p_2 e_2+p e_1) &=& \left\lbrace
\begin{array}{ll}
H^L((0,z_2),p_2 e_2+p e_1) & \mbox{ if } p \le p^{-,L}_0(z,p_2),\\
E^L_0(z_2,p_2)  & \mbox{ if } p > p^{-,L}_0(z,p_2),
\end{array}
\right.    \\
H^{+,L}((0,z_2),p_2 e_2+p e_1) &=& \left\lbrace
\begin{array}{ll}
E^L_0(z_2,p_2)  & \mbox{ if } p \le p^{+L}_0(z,p_2),\\
H^L((0,z_2),p_2 e_2+p e_1) & \mbox{ if } p > p^{+,L}_0(z,p_2).
\end{array}
\right.
  \end{eqnarray*}
\item For  $z_2, p_2\in \R$,  there exist two  unique real numbers $p^{+,R}_0(z_2,p_2)\le  p^{-,R}_0(z_2,p_2) $ such that
  \begin{eqnarray*}
H^{-,R}((0,z_2),p_2 e_2+p e_1) &=& \left\lbrace
\begin{array}{ll}
E^R_0(z_2,p_2)  & \mbox{ if } p \le p^{-,R}_0(z,p_2),\\
H^R((0,z_2),p_2 e_2+p e_1) & \mbox{ if } p > p^{-,R}_0(z,p_2),
\end{array}
\right.\\
H^{+,R}((0,z_2),p_2 e_2+p e_1) &=& \left\lbrace
\begin{array}{ll}
H^R((0,z_2),p_2 e_2+p e_1) & \mbox{ if } p \le p^{+,R}_0(z,p_2),\\
E^R_0(z_2,p_2)  & \mbox{ if } p > p^{+,R}_0(z,p_2).
\end{array}
\right.
  \end{eqnarray*}
  \end{enumerate}
\end{lemma}

\section{A new Hamiltonian involved in the effective transmission condition}
\label{sec:new-hamilt-involv}
In this section, we construct the effective Hamiltonian corresponding to effective dynamics staying on the interface $\Gamma$, 
by using similar ideas as those presented in \cite{MR3299352}. 
 We will define an effective Hamiltonian $E$ on $\Gamma$ as the  limit of a sequence of  ergodic constants for  state-constrained
 problems in larger and larger truncated domains.
 We will also construct correctors associated to the effective Hamiltonian. 
The  noteworthy difficulty is that the correctors need to be defined in an unbounded domain.
\subsection{Fast and slow variables}
Let us introduce the fast variable  $y_2 = \frac{z_2}{\epsilon}$. Neglecting the contribution of  $\epsilon g(y_2)$  in 
the Hamiltonians  $\tilde H_\epsilon^i$ previously defined in \eqref{eq:7new}, we obtain the new Hamiltonians
$\tilde H^i:\R^2\times \R^2\times \R\to \R$:  
\begin{equation}
\label{explicitelyfast}
\tilde{H}^i(z,p,y_2):= \max_{a\in A^i} \left(-{\tilde{J}}(y_2) f^i(z,a)\cdot
p  -\ell^i(z,a)\right),
\end{equation} 
where
\begin{equation}
\label{jacobiany}
{\tilde{J}}(y_2)=\left(\begin{matrix}
1 &-g'(y_2)\\
0 &1
\end{matrix}\right).
\end{equation}

As above, using $\sigma^i$  introduced in \S~\ref{sec:geometry}, we also define  $\tilde{H}^{+,i}$ and $\tilde{H}^{-,i}$: $\R^2\times\R^2\times \R\to \R$ by
\begin{equation}
\label{Ham+fastslowGamma}
\tilde{H}^{\pm,i} (z,p, y_2)=\max_{\begin{array}{c}
a\in A^i \hbox{ s.t. } \\
 \pm \sigma^i{\tilde{J}}(y_2) f^i(z,a)\cdot e_1\ge 0
\end{array}
}  \left(-{\tilde{J}}(y_2)f^i(z,a)\cdot p  -\ell^i(z,a)\right).
\end{equation}
\begin{lemma}
   \label{sec:fast-slow-variables-2}
With $L_f$, $M_f$, $M_\ell$ and $\delta_0$     appearing in Assumptions  $[\rm{H}0]$-$[\rm{H}3]$, 
for any $z\in \R^2$, $ y_2\in \R$ and  $p,p'\in \R^2$, 
\begin{equation}\label{eq:12}
| \tilde H^{i}(z,p,y_2)- \tilde H^{i}(z,p',y_2) | \le M_f |p-p'|, \quad i=L,R,
\end{equation}
and there exists a constant $M>0$ (which can be computed from $L_f$, $M_f$, $\delta_0$ and  $\|g'\|_{\infty}$)  and a modulus of continuity $\omega$ (which can be deduced from $\omega_\ell$, $M_\ell$ $L_f$, $\delta_0$ and  $\|g'\|_{\infty}$) such that  
for any $z,z'\in \R^2$, $y_2\in\R$ and $p\in \R^2$,
\begin{equation}\label{eq:13}
\mid \tilde H^{i}(z,p,y_2)- \tilde H^{i}(z',p,y_2) \mid \le M |p||z-z'|+\omega(|z-z'|).
\end{equation}
Similar estimates hold for  $\tilde H^{+,i}$ and $\tilde H^{-,i}$.
\end{lemma}
\begin{proof}
 The proof is standard for the Hamiltonians $\tilde H^i$. Adapting the proofs of Lemmas 3.5 and 3.6 in \cite{oudet2014}, 
we see that  similar estimates hold for  $\tilde H^{+,i}$ and $\tilde H^{-,i}$: 
the proofs are not direct and rely on the convexity of $\{ (\tilde J (y_2) f^i (z, a), \ell^i (z,  a)), a\in A^i\}$, on the fact that the dynamics $\tilde J (y_2) f^i (z, \cdot)$ satisfy  a strong controlability property  similar to  $[\rm{H}3]$ uniformly w.r.t. $z$,   and on continuity properties of $ (z, a)\mapsto \tilde J(y_2) f^i (z, a)$ and  $ (z, a)\mapsto \ell^i (z, a)$ similar to  $[\rm{H}0]$ and $[\rm{H}1]$.
\end{proof}

\begin{remark}
  \label{sec:fast-slow-variables}
Take $z\in \Omega^i$ and $p\in\R^2$.  
The unique real number $\lambda^i(z, p)$ such that
the following one dimensional cell problem in the variable $y_2$
 \begin{equation}
\label{eq:cell_pb_in_omega_i}
\left\lbrace
\begin{array}{ll}
\tilde{H}^i(z,p+\chi'(y_2)e_2, y_2)=\lambda^i(z, p),\\
\chi \mbox{ is 1-periodic w.r.t. }y_2,
\end{array}
\right.
\end{equation}
admits a viscosity solution is $\lambda^i(z, p)=H^i(z,p)$; indeed, for this choice of  $\lambda^i(z, p)$,   it is easy to check that 
$\chi(y_2)= p_1 g(y_2)$ is a  solution of \eqref{eq:cell_pb_in_omega_i} and the uniqueness of $\lambda^i(z, p)$ such that \eqref{eq:cell_pb_in_omega_i}
 has a solution is  well known, see e.g.  \cite{LPV,MR1007533}.
\end{remark}
\begin{remark}
  \label{sec:fast-slow-variables-1}
 For any $(0,z_2)\in \Ga$, $p\in \R^2$ and $y_2\in \R$,     the functions $p_1\mapsto \tilde H^{\pm,i}((0,z_2),p+p_1e_1,y_2)$  have the same  monotonicity properties
as those  stated in point 1 in  Lemma~\ref{sec:effect-hamilt-gamma-1} for  $p_1\mapsto H^{\pm,i}((0,z_2),p+p_1e_1)$. Similarly,  
one can prove the  counterparts of points 2 and 3 in  Lemma~\ref{sec:effect-hamilt-gamma-1} involving $\tilde H^{\pm,i}((0,z_2),p+p_1e_1,y_2)$,  $\tilde H^{i}((0,z_2),p+p_1e_1,y_2)$ and
\[\tilde E_0^i(z_2,p_2,y_2)=\min\left\lbrace \tilde H^i((0, z_2),p_2e_2+qe_1,y_2), \; q\in \R \right\rbrace.\]
\end{remark}

\subsection{Ergodic constants for state-constrained problems in truncated domains}
\label{sec:ergod-const-state}

\subsubsection{State-constrained problem in truncated domains}\label{sec:state-constr-probl}
Let us fix $z=(0,z_2)\in \Ga$ and $p_2\in \R$. For $\rho>0$, we consider the {\sl truncated cell problem}:
\begin{equation}
\label{trunc-cellp}
\left\{
    \begin{array}[c]{lll}
    \tilde{H}^L((0,z_2),Du(y)+p_2e_2,y_2)&=\lambda_\rho(z_2, p_2)&\hbox{ in } (-\rho,0)\times \R,  \\
        \tilde{H}^R((0,z_2),Du(y)+p_2e_2,y_2)&=\lambda_\rho(z_2, p_2)&\hbox{ in } (0,\rho)\times \R,  \\
\ds \max_{i=L,R}\left( \tilde{H}^{+,i} ((0,z_2),Du^i(y)+p_2e_2,y_2)\right) &=\lambda_\rho(z_2, p_2)&\hbox{ on } \Ga,\\
\tilde{H}^{-,L} ((0,z_2),Du(y)+p_2e_2,y_2) &=\lambda_\rho(z_2, p_2)&\hbox{ on } \{-\rho\}\times \R ,\\
\tilde{H}^{-,R} ((0,z_2),Du(y)+p_2e_2,y_2) &=\lambda_\rho(z_2, p_2)&\hbox{ on } \{\rho\}\times \R ,\\
    u \hbox{ is 1-periodic w.r.t. } y_2,
    \end{array}
\right.
\end{equation}
where the Hamiltonians $\tilde H^i$, $\tilde{H}^{+,i}$ and $\tilde{H}^{-,i}$ are respectively defined in \eqref{explicitelyfast} and \eqref{Ham+fastslowGamma}.
The notions of viscosity subsolution, supersolution and solution of \eqref{trunc-cellp} are defined in the same way as in Definition \ref{netviscoa} using the set of test-functions
\begin{equation}
\label{def:viscosity_solution_of_trunc-cell_pb}
\cR_\rho = 
\left\lbrace \psi|_{[-\rho,\rho]\times\R}, \; \psi \in \cR
\right\rbrace,
\end{equation}
with $\cR$ defined in Definition~\ref{def:set_test_function_with_straight_geometry}. The following stability property allows one to construct a solution of \eqref{trunc-cellp}:
\begin{lemma}[A stability result]
\label{lem:stability_for_truncated_cell_problem}
Let $(u^\eta)_{\eta}$ be a sequence of uniformly Lipschitz continuous solutions of the perturbed equation
\begin{equation}
\label{trunc-cellpapprox}
\left\{
    \begin{array}[c]{lll}
   \eta u(y)+\tilde{H}^L((0,z_2),Du(y)+p_2e_2,y_2)&=\lambda_\eta&\hbox{ in } (-\rho,0)\times \R,  \\
\eta u(y)+\tilde{H}^R((0,z_2),Du(y)+p_2e_2,y_2)&=\lambda_\eta&\hbox{ in } (0,\rho)\times \R,  \\
\ds \eta u(y)+ \max_{i=L,R}\left( \tilde{H}^{+,i} ((0,z_2),Du^i(y)+p_2e_2,y_2)\right) &=\lambda_\eta&\hbox{ on } \Ga,\\
\eta u(y)+\tilde{H}^{-,L} ((0,z_2),Du(y)+p_2e_2,y_2) &=\lambda_\eta&\hbox{ on } \{-\rho\}\times \R ,\\
\eta u(y)+\tilde{H}^{-,R} ((0,z_2),Du(y)+p_2e_2,y_2) &=\lambda_\eta &\hbox{ on } \{\rho\}\times \R ,\\
    u \hbox{ is 1-periodic w.r.t. } y_2,
    \end{array}
\right.
\end{equation}
such that $\lambda_\eta$ tends to $\lambda$ as $\eta$ tends to $0$ and $u^\eta$ converges to $u^0$  uniformly in $[-\rho,\rho]\times \R$. Then $u^0$ is a viscosity solution of  \eqref{trunc-cellp}  (replacing  $\lambda_\rho(z_2,p_2)$ with  $\lambda$).
\end{lemma}
\begin{proof}
The proof of Lemma \ref{lem:stability_for_truncated_cell_problem} follows the lines of the proofs of Theorem 6.1 and Theorem 6.2 in \cite{achdou:hal-00847210}. Actually, the proof is even simpler in the present case since the involved Hamiltonians do not depend of $\eta$. We give it in Appendix \ref{appendix:proof_stability_truncated_cell_pb} for the reader's convenience.
\end{proof}
The following  comparison principle for \eqref{trunc-cellpapprox} yields the uniqueness of the constant $\lambda_\rho(z_2,z_2)$ for which the cell-problem \eqref{trunc-cellp} admits a solution:
\begin{lemma}[A comparison result]
\label{th:comparison_truncated_cell_pb}
 For $\eta>0$, let $u$ be a  bounded subsolution of \eqref{trunc-cellpapprox}  and 
$v$ be a  bounded supersolution of \eqref{trunc-cellpapprox}. Then $u\le v$ in $[-\rho,\rho]\times \R$.
\end{lemma}
\begin{proof}
As for Theorem \ref{sec:comparison-principle-1}, this result can be obtained by applying \cite[Theorem~3.4]{oudet2014}.
\end{proof}

\begin{lemma}
\label{lem:existence_solution_truncated_cell_pb}
There is a unique  $\lambda_\rho(z_2, p_2)\in \R$ such that \eqref {trunc-cellp} admits a bounded solution. 
For this choice of  $\lambda_\rho(z_2, p_2)$, there exists a  solution  $\chi_\rho(z_2,p_2,\cdot)$ which is Lipschitz continuous with  Lipschitz constant
 $L$ depending on $p_2$ only (independent of $\rho$).
\end{lemma}
\begin{proof} With the set $\cM$ defined in (\ref{eq: def-M}), let us consider the new {\sl freezed} dynamics $ f_{z_2}: \cM\to \R^2$ and running costs $ \ell_{z_2,p_2}: \cM\to \R^2$:
  \begin{eqnarray}
\label{eq:def_dynamic_pb_intermediate}
  f_{z_2}(y, a)&=&\left\{
    \begin{array}[c]{l}
      \left(\begin{matrix}
1 & -g'(y_2) \\
0 & 1
\end{matrix}\right)
f^L((0,z_2),a) \quad  \hbox{ if } y_1\le 0, a\in A^L,
\\
 \left(\begin{matrix}
1 & -g'(y_2) \\
0 & 1
\end{matrix}\right)
f^R((0,z_2),a)  \quad \hbox{ if } y_1\ge 0 , a\in A^R, 
    \end{array}\right.
\\
\label{def:ell_pb_oc_truncated}
 \ell_{z_2,p_2}(y, a)&=&
\left\{
    \begin{array}[c]{l}
  f ^L_2((0, z_2),a)p_2+ \ell^L((0, z_2),a)\quad   \hbox{ if } y_1\le 0, a\in A^L, \\
  f ^R_2((0, z_2),a)p_2+ \ell^R((0, z_2),a)  \quad \hbox{ if } y_1\ge 0, a\in A^R,
\end{array}
\right.
\end{eqnarray}
where $f_2$ stands for the second component of $f$.\\
Let $\cT_{z_2,x,\rho}$ be the set of admissible trajectories starting from $y\in (-\rho,\rho)\times \R$ and constrained to $[-\rho,\rho]\times \R$:
\begin{equation}
  \label{eq:2freezed}
\cT_{z_2,y,\rho}=\left\{
  \begin{array}[c]{ll}
    (\zeta_y,a)  \ds \in L_{\rm{loc}}^\infty( \R^+; \cM) : \quad   & \zeta_y\in {\rm{Lip}}(\R^+; [-\rho,\rho]\times \R), 
 \\  &\ds  \zeta_y(t)=y+\int_0^t f_{z_2}( \zeta_y(s), a(s)) ds \quad  \forall t\in \R^+ 
  \end{array}\right\}.
\end{equation}
For any $\eta>0$, the cost associated to the trajectory $(\zeta_y,a)\in\cT_{z_2,y,\rho}$ is
\begin{equation}
\label{cost_functional_R}
 \cJ^\eta_\rho(z_2, p_2, y;( \zeta_y, a) )=\int_0^\infty \ell_{z_2,p_2}(\zeta_y(t), a(t)) e^{-\eta t} dt,
\end{equation}
and we introduce the optimal control problem:
\begin{equation}
\label{valueR}
v^\eta_\rho(z_2, p_2,y)=\inf_{(\zeta_y,a)\in\cT_{z_2,y,\rho}} \cJ^\eta_\rho(z_2, p_2,y;( \zeta_y, a) ).
\end{equation}
Thanks to $[\rm{H}3]$,  we see that if $\delta_0'=\frac{\delta_0}{ \sqrt 2 (1+\parallel g' \parallel_\infty)}$, then $B(0,\delta_0')\subset \{f_{z_2}(y, a), a\in A^i\}$ for  any $i=L,R$, 
$y\in [-\rho,\rho]\times \R$. This strong controlability property can be proved
 in the same manner as  $[\tH3]_\epsilon$ in $\S$ \ref{subsec:straight_optima_control}.  From this, it follows that
for any $y,y'\in [-\rho,\rho]\times \R$,
\begin{equation}
\label{eq:v_R^rho_L_lipschitz}
\vert v^\eta_\rho(z_2, p_2,y)-v^\eta_\rho(z_2, p_2,y')\vert\leq L(p_2)\vert y - y'\vert.
\end{equation}
for some $ L(p_2)=  L_1 + L_2 |p_2|$ with $L_1$, $L_2$ depending on $M_f$, $M_\ell$ $ \delta_0$ and $\parallel g'\parallel_\infty$ but not on $p_2$.
Introducing $\chi_\rho^\eta(z_2, p_2, y)=v^\eta_\rho(z_2, p_2,y)-v^\eta_\rho(z_2, p_2,(0,0))$, 
we deduce from \eqref{valueR} and \eqref{eq:v_R^rho_L_lipschitz} that there exists a constant $K=K(p_2)$ such that 
\begin{equation}
\label{eq: estimate_pho_v}
\begin{array}{l}
\vert \eta v^\eta_\rho(z_2, p_2,y)\vert\leq K,\quad \hbox{and}\quad 
\vert \chi_\rho^\eta(z_2, p_2, y)\vert\leq K.
\end{array}
\end{equation}
From Ascoli-Arzela's theorem, up to the extraction a subsequence, $\chi_\rho^\eta(z_2, p_2, \cdot)$ and $-\eta v^\eta_\rho(z_2, p_2,\cdot)$ converge uniformly  respectively to a Lipschitz function $\chi_\rho(z_2,p_2,\cdot)$ defined on $[-\rho,\rho]\times \R$ (with Lipschitz constant $L$) and to a constant $\lambda_\rho(z_2, p_2)$ as $ \eta\rightarrow 0$.\\
On the other hand, with the  arguments contained in \cite{MR3057137,imbert:hal-00832545,imbert:hal-01073954}, it can be proved that  $v_\rho^\eta(z_2, p_2, \cdot)$ is a viscosity solution of \eqref{trunc-cellpapprox} with $\lambda_\eta=0$. Hence, $\chi_\rho^\eta(z_2, p_2, \cdot)$ is a sequence of viscosity solutions of \eqref{trunc-cellpapprox} for $\lambda_\eta=-\eta v_\rho^\eta(z_2,p_2,(0,0))$, and $\lambda_\eta \to \lambda_\rho(z_2,p_2)$ as $\eta$ tends to $0$. From the stability result in Lemma \ref{lem:stability_for_truncated_cell_problem}, the function $\chi_\rho(z_2,p_2,\cdot)$ is a viscosity solution of \eqref{trunc-cellp}.
Finally, uniqueness can be proved in a classical way using the comparison principle in Lemma \ref{th:comparison_truncated_cell_pb} and the boundedness of $\chi_\rho$.
\end{proof}
\subsubsection{Passage to the limit as $\rho\to +\infty$}
\label{sec:passage-limit-as}
By definition of $\cT_{z_2,y,\rho}$,  it is clear that if $\rho_1\leq \rho_2$, then $\cT_{z_2,y,\rho_1}\subset \cT_{z_2,y,\rho_2}$.
 Then, thanks to \eqref{valueR} and \eqref{eq: estimate_pho_v}, we see that
\begin{equation*}
-\eta v^\eta_{\rho_1}\leq -\eta v^\eta_{\rho_2}\le K,
\end{equation*}
and letting $\eta \rightarrow 0$, we obtain that 
\begin{equation}
\lambda_{\rho_1}(z_2, p_2)\leq \lambda_{\rho_2}(z_2, p_2)\leq K.
\end{equation}
\begin{definition}
We define the effective tangential Hamiltonian $E(z_2, p_2)$ as 
\begin{equation}
\label{eq:def_E}
E(z_2, p_2)=\lim_{\rho\rightarrow \infty} \lambda_{\rho}(z_2, p_2).
\end{equation}
\end{definition}
For $z_2,p_2\in\R$ fixed, we consider the  {\sl global cell-problem}
\begin{equation}
\label{cellpE}
\left\{
    \begin{array}[c]{lll}
    \tilde{H}^i((0,z_2),Du(y)+p_2e_2,y_2)&=E(z_2, p_2)&\hbox{ in } \Omega^i,  \\
\max\left( \tilde{H}^{+,L} ((0,z_2),Du^L(y)+p_2e_2,y_2),
\tilde{H}^{+,R}((0,z_2),Du^R(y)+p_2e_2,y_2)  \right)&=E(z_2, p_2)&\hbox{ on } \Ga,\\
 u \hbox{  is  1-periodic w.r.t.  } y_2,
    \end{array}
\right.
\end{equation}
The following stability result is useful for proving  the existence of a viscosity solution $u$ of the cell-problem \eqref{cellpE}:
\begin{lemma}
\label{thm:stability_from_truncated_cell_pb_to_global_cell_pb}
Let $u_\rho$ be a sequence of uniformly Lipschitz continuous solutions  of the truncated cell-problem \eqref{trunc-cellp} which converges to $u$ 
locally uniformly on $\R^2$. Then $u$ is a viscosity solution of the global cell-problem \eqref{cellpE}.
\end{lemma}
\begin{proof}
Proceed exactly in the same way as in the proof of Lemma \ref{lem:stability_for_truncated_cell_problem}.
\end{proof}
\begin{theorem}[Existence of a global corrector]
\label{lem:existence_cellpE}
There exists $\chi(z_2, p_2,\cdot)$ a Lipschitz continuous viscosity solution of\eqref{cellpE}  with the same Lipschitz constant $L$ as in \eqref{eq:v_R^rho_L_lipschitz} and
such that $\chi(z_2, p_2,(0,0))=0$.
\end{theorem}
\begin{proof}
Let $\chi_\rho(z_2,p_2,\cdot)$ be the sequence of solutions of \eqref{trunc-cellp} given by Lemma \ref{lem:existence_solution_truncated_cell_pb}. 
Recall that $\chi_\rho(z_2,p_2,\cdot)$ is Lipschitz continuous with  Lipschitz constant $L$ independent of $\rho$ and periodic with respect to $y_2$.
By taking $\chi_\rho(z_2,p_2,\cdot)-\chi_\rho(z_2,p_2,(0,0))$ instead of $\chi_\rho(z_2,p_2,\cdot)$, we may assume that $\chi_\rho(z_2,p_2,(0,0))=0$.
 Thus, $\chi_\rho(z_2,p_2,\cdot)$ is locally bounded and thanks to Ascoli-Arzela's theorem, up to the extraction a subsequence, $\chi_\rho(z_2,p_2,\cdot)$ converges locally uniformy to a  function $\chi(z_2, p_2,\cdot)$, which is Lipschitz continuous and periodic with respect to $y_2$ and satisfies  $\chi(z_2, p_2, (0, 0))=0$.
 Thanks to the stability result in Lemma \ref{thm:stability_from_truncated_cell_pb_to_global_cell_pb}, $\chi(z_2, p_2,\cdot)$ is a viscosity solution of \eqref{cellpE}.
\end{proof}

 \subsubsection{Comparison between  $E_0$ and $E$ respectively defined in (\ref{def:E_0}) and (\ref{eq:def_E})}
For $\epsilon>0$,  let us call $W_\epsilon(z_2, p_2,y)=\epsilon\chi(z_2, p_2,\frac y {\epsilon})$.
 The following result is reminiscent of \cite[Theorem 4.6,iii]{galise:hal-01010512}:
 \begin{lemma}
\label{lem:rescaling_omega} 
 For any $z_2, p_2\in \R$,   
there exists a subsequence    $\epsilon_n$  such that  $W_{\epsilon_n}  (z_2, p_2,\cdot)$ converges locally uniformly to a 
Lipschitz  function $y\mapsto W(z_2, p_2, y)$,  with the Lipschitz constant $L$ appearing in \eqref{eq:v_R^rho_L_lipschitz}. 
This function is constant with respect to $y_2$ and satisfies $W(z_2, p_2,0)=0$. It is  a viscosity solution of 
\begin{equation}
\label{W}
H^i((0,z_2),Du(y)+p_2e_2)=E(z_2, p_2)\quad\hbox{ in } \Omega^i.
\end{equation}
\end{lemma}
\begin{proof}
It is clear that $y\mapsto W_\epsilon(z_2, p_2,y)$ is a Lipschitz continuous function with constant $L$ and that $W_\epsilon(z_2, p_2,(0,0))=0$.
Thus, from Ascoli-Arzela's Theorem, we may assume that  $y\mapsto W_\epsilon(z_2, p_2,y)$  converges locally uniformly to some function $y \mapsto W(z_2,p_2,y)$,  up to the extraction of subsequences. The function $y \mapsto W(z_2,p_2,y)$ is Lipschitz continuous  with constant $L$ and  $W(z_2,p_2,(0,0))=0$. Moreover, since $W_\epsilon(z_2, p_2,y)$ is periodic with respect to $y_2$  with period $\epsilon$, $W(z_2,p_2,y)$ does not depend on $y_2$. To prove that $W(z_2,p_2,\cdot)$ is a viscosity solution of \eqref{W},
we first observe   that $y\mapsto W_\epsilon(z_2, p_2,\cdot)$ is a viscosity solution of
\begin{equation}
\label{cellpEepsilon}
    \tilde{H}^i((0,z_2),Du(y)+p_2e_2,\frac {y_2} {\epsilon})=E(z_2, p_2)\quad \hbox{ in } \Omega^i.
\end{equation}
For $i=L,R$, assume that  $\bar y\in \Omega^i$, $\phi \in \cC^1(\Omega^i)$ and $r_0<0$ are such that  $ B(\bar y,r_0)\subset \Omega^i$  and that 
\begin{equation*}
\label{proof:rescaling1}
W(z_2,p_2,y)-\phi(y)<W(z_2,p_2,\bar y)-\phi(\bar y)=0 \mbox{ for  } y\in B(\bar y,r_0)\setminus\{\bar y\}.
\end{equation*}
We wish to prove that
$H^i((0,z_2),D\phi(\bar y)+p_2e_2)\le E(z_2, p_2)$.
Let us argue by contradiction and assume that there exists $\theta>0$ such that
\begin{equation}
\label{proof:rescaling3}
H^i((0,z_2),D\phi(\bar y)+p_2e_2)= E(z_2, p_2)+\theta.
\end{equation}
Take $\phi_\epsilon(y)=\phi(y)+\epsilon\partial_{y_1}\phi(\bar y)g(\frac{y_2}{\epsilon})-\delta$, where $\delta>0$ is a fixed positive number.
 We claim that for $\epsilon>0$ and $r>0$ small enough, $\phi_\epsilon$ is a viscosity supersolution of
\begin{equation}
\label{proof:rescaling4}
\tilde H^i((0,z_2),Du(y)+p_2e_2,\frac{y_2}{\epsilon})\ge E(z_2, p_2)+\frac{\theta}{2} \quad \quad \hbox{ in } B(\bar y,r).
\end{equation}
Indeed $\phi_\epsilon$ is a regular function which satisfies
\begin{displaymath}
\tilde H^i((0,z_2),D\phi_\epsilon(y)+p_2e_2,\frac{y_2}{\epsilon})=H^i\left((0,z_2),D\phi(y)+g'(\frac{y_2}{\epsilon})(\partial_{y_1}\phi(\bar y)-\partial_{y_1}\phi(y))e_2\right),
\end{displaymath}
and we deduce \eqref{proof:rescaling4} from \eqref{proof:rescaling3} and the regularity properties of the Hamiltonian $H^i$.
Hence,  $W_\epsilon(z_2,p_2,\cdot)$ is a subsolution of \eqref{cellpEepsilon} and  $\phi_\epsilon$ is a supersolution of \eqref{proof:rescaling4} in $B(\bar y,r)$.
  Moreover for $r>0$ small enough,
$ \max_{y\in \partial B(\bar y,r)}\left(W(z_2,p_2,y)-\phi(y) \right)<0$. Hence, for $\epsilon>0$ small enough
$\max_{y\in \partial B(\bar y,r)}\left(W_\epsilon(z_2,p_2,y)-\phi_\epsilon(y) \right) \le 0$.\\
Thanks to a standard comparison principle (which holds thanks to the fact that $\frac{\theta}{2}>0$)
\begin{equation}
\label{proof:rescaling6}
\max_{y\in B(\bar y,r)}\left(W_\epsilon(z_2,p_2,y)-\phi_\epsilon(y) \right) \le 0.
\end{equation}
 Letting $\epsilon \to 0$ in \eqref{proof:rescaling6}, we deduce that $W(z_2,p_2,\bar y)\le\phi(\bar y)-\delta$,
which is in contradiction with the assumption.
\end{proof}
 Using Lemma~\ref{lem:rescaling_omega}, it is possible to  compare $E_0(z_2, p_2)$  and $E(z_2, p_2)$  respectively defined in (\ref{def:E_0}) and (\ref{eq:def_E}):
 \begin{proposition}
\label{cor:E_bigger_than_E_0}
For any $z_2,p_2 \in \R$,
\begin{equation}
\label{eq:cor:E_bigger_than_E_0}
E(z_2, p_2)\geq  E_0(z_2, p_2).
\end{equation}
\end{proposition}
\begin{proof}
 Let $i\in \{L,R\}$ be fixed. Thanks to Lemma \ref{lem:rescaling_omega}, the function $y\mapsto W(z_2,p_2,y)$ is a viscosity solution of \eqref{W} in $\Omega^i$.
Therefore,   \eqref{W} is satisfied by $W (z_2,p_2,\cdot)$ almost everywhere. Keeping in mind that  $W(z_2,p_2,y)$ is independent of $y_2$, we see that  for almost all $y\in \Omega^i$, $E(z_2, p_2)=H^i((0,z_2),\partial_{y_1}W(z_2,p_2,y_1)e_1+p_2e_2)\ge E_0^i(z_2,p_2)$.
\end{proof}
From Proposition \ref{cor:E_bigger_than_E_0} and the coercivity of the Hamiltonian $H^i$, the following numbers are well defined for all  $z_2, p_2 \in \R$:
\begin{eqnarray}
 \label{eq:5}
\overline{\Pi}^L(z_2,p_2)\!= \!\max\left\lbrace p\in \R : H^L((0,z_2), p_2e_2+pe_1)=H^{-,L}((0,z_2), p_2e_2+pe_1)=E(z_2,p_2) \right\rbrace\\
\label{eq:6}
\widehat{\Pi}^L(z_2,p_2)\!= \!\min\left\lbrace p\in \R : H^L((0,z_2), p_2e_2+pe_1)= H^{-,L}((0,z_2), p_2e_2+pe_1)=E(z_2,p_2) \right\rbrace\\
\label{eq:9}
\overline{\Pi}^R(z_2,p_2)\!= \!\min\left\lbrace p\in \R : H^R((0,z_2), p_2e_2+pe_1)=H^{-,R}((0,z_2), p_2e_2+pe_1)=E(z_2,p_2) \right\rbrace\\
\label{eq:10}
\widehat{\Pi}^R(z_2,p_2)\!= \!\max\left\lbrace p\in \R : H^R((0,z_2), p_2e_2+pe_1)=H^{-,R}((0,z_2), p_2e_2+pe_1)=E(z_2,p_2) \right\rbrace
\end{eqnarray}
\begin{remark}\label{sec:comp-betw-e_0}
In \S~\ref{sec:furth-results-corr}, see in particular Corollaries \ref{cor:slopes_omega} and \ref{cor:control_slopes_W} below,  we will see that the function $W$ which is defined in Lemma~\ref{lem:rescaling_omega} and  provides information on the   growth of $y\mapsto \chi(z_2,p_2,y)$ as $|y_1|\to \infty$, satisfies
\begin{equation*}
\overline{\Pi}^L(z_2,p_2) y_1 1_{\Omega^L}+\overline{\Pi}^R(z_2,p_2) y_1 1_{\Omega^R}\le W(z_2,p_2,y) \le  \widehat{\Pi}^L(z_2,p_2) y_1 1_{\Omega^L}+ \widehat{\Pi}^R(z_2,p_2) y_1 1_{\Omega^R}.
\end{equation*}    
These growth properties at infinity show that $\chi(z_2,p_2,\cdot)$ is precisely the corrector associated to the reduced set of test-functions proposed by Imbert and Monneau in \cite{imbert:hal-00832545,imbert:hal-01073954}, see \S~\ref{def:test_functions_set_restricted} below.
\end{remark}

\begin{remark}
\label{rmk:equality_bar_Pi_and_hat_Pi}
From the  convexity  of the Hamiltonians $H^i$ and $H^{-,i}$, we deduce that if $E^i_0(z_2,p_2)<E(z_2,p_2)$,  then 
$\overline{\Pi}^i(z_2,p_2)=\widehat{\Pi}^i(z_2,p_2)$. In this case, we will use the notation
\begin{equation}
\label{eq:special_notation_for_Pi}
\Pi^i(z_2,p_2)=\overline{\Pi}^i(z_2,p_2)=\widehat{\Pi}^i(z_2,p_2).
\end{equation}
\end{remark}

 \section{Further properties of the correctors}
\label{sec:furth-results-corr}
In this section, we prove further growth properties of the correctors, which will be useful in the proof of convergence in \S~\ref{sec:proof-main-result} below, see Remark~\ref{sec:comp-betw-e_0}.
We start by stating a useful comparison principle related to a mixed boundary value problem:
 \begin{lemma}
   \label{sec:furth-prop-corr}
Take $0<\rho_1<\rho_2$, $z_2,p_2,\lambda\in \R$, a continuous function $U_0: \R\to \R$ and  $\epsilon_0>0$.
Let $v$ be a  continuous viscosity supersolution of 
\begin{equation}
\label{eq:def_mixed_boundary_value_problem_super}
\left\lbrace
\begin{array}{ll}
    \tilde{H}^R((0,z_2),Dv(y)+p_2e_2,y_2)  \ge \lambda, & y=(y_1,y_2)\in (\rho_1,\rho_2)\times \R,\\
      \tilde{H}^{-,R}((0,z_2),Dv(y)+p_2e_2,y_2)  \ge \lambda, & y=(y_1,y_2)\in \{\rho_2\}\times \R,\\
      v(y)  \ge U_0(y_2), & y=(y_1,y_2)\in \{\rho_1\}\times \R,\\
      v \hbox{ is 1-periodic w.r.t. } y_2,
\end{array}
\right.
\end{equation}
and  $u$ be a continuous  viscosity subsolution of 
\begin{equation}
\label{eq:def_mixed_boundary_value_problem_sub}
\left\lbrace
\begin{array}{ll}
    \tilde{H}^R((0,z_2),Du(y)+p_2e_2,y_2)  \le \lambda-\epsilon_0, & y=(y_1,y_2)\in (\rho_1,\rho_2)\times \R,\\
      \tilde{H}^{-,R}((0,z_2),Du(y)+p_2e_2,y_2)  \le \lambda-\epsilon_0, & y=(y_1,y_2)\in \{\rho_2\}\times \R,\\
      u(y)  \le U_0(y_2), & y=(y_1,y_2)\in \{\rho_1\}\times \R,\\
      u \hbox{ is 1-periodic w.r.t. } y_2,
\end{array}
\right.
\end{equation}
where  the inequalities on $y_1=\rho_1$  are understood  pointwise.
Then, $u\le v$ in $[\rho_1,\rho_2]\times \R$.
 \end{lemma}
The proof is rather classical, and follows the lines of  \cite{MR1484411} Theorem IV.5.8. We skip it for brevity.
 \begin{remark}
   \label{sec:furth-results-corr-1}
From \cite[Proposition~2.14]{imbert:hal-00832545}, we know  that a bounded lsc function $v$ is a supersolution of \eqref{eq:def_mixed_boundary_value_problem_super} if and only if it is a supersolution of
 \begin{displaymath}
 \left\lbrace
\begin{array}{ll}
    \tilde{H}^R((0,z_2),Dv(y)+p_2e_2,y_2)  \ge \lambda, & y=(y_1,y_2)\in (\rho_1,\rho_2]\times \R,\\
      v(y)  \ge U_0(y_2), & y=(y_1,y_2)\in \{\rho_1\}\times \R,\\
      v \mbox{ is 1-periodic w.r.t. } y_2,
\end{array}
\right.
\end{displaymath}
and that a bounded usc function $u$ is a subsolution of \eqref{eq:def_mixed_boundary_value_problem_sub} if and only if it is a subsolution of
\begin{displaymath}
 \left\lbrace
\begin{array}{ll}
    \tilde{H}^R((0,z_2),Du(y)+p_2e_2,y_2)  \le \lambda-\epsilon_0, & y=(y_1,y_2)\in (\rho_1,\rho_2)\times \R,\\
      u(y)  \le U_0(y_2), & y=(y_1,y_2)\in \{\rho_1\}\times \R,\\
      u \mbox{ is 1-periodic w.r.t. } y_2.
\end{array}
\right.
\end{displaymath}
In other words, the boundary conditions on $y_1=\rho_2$ correspond to state constraints.
 \end{remark}

 \begin{remark}
 \label{rmk:about_mixed_boundary_value_problem}
 There is of course a similar result for the mirror boundary value problem posed in
 $[-\rho_2,-\rho_1]\times \R \subset \Omega^L$  with  the Hamiltonian $\tilde{H}^L$ instead of $\tilde{H}^R$,  a Dirichlet condition on  $y_1=-\rho_1$ and a state constrained boundary condition on $y_1=-\rho_2$ (i.e. involving $\tilde{H}^{-,L}$).
 \end{remark}
\begin{proposition}[Control of slopes on the truncated domain]
\label{SLOPESLemma1}
 With $E$ and $E^R_0$ respectively defined in~\eqref{eq:def_E} and \eqref{def:E_0^i},  let $z_2,p_2\in\R$ be such that 
$E(z_2, p_2)>E_0^R(z_2,p_2)$.
There exists $\rho^*=\rho^*(z_2, p_2)>0$, $\delta^*=\delta^*(z_2,p_2)>0$, \mbox{$m(\cdot)=m(z_2,p_2,\cdot):\R_+ \to \R_+$} satisfying $\lim_{\delta \to 0^+}m(\delta)=0$ and $M^*=M^*(z_2,p_2)$ such that for all  $\delta\in (0,\delta^*]$, $\rho\ge \rho^*$, $y=(y_1,y_2)\in [\rho^*,\rho]\times \R$, $h_1\in [0,\rho-y_1]$ and $h_2\in \R$,
\begin{equation}
\label{slope3}
\chi_\rho(z_2, p_2, y+h_1e_1+h_2 e_2)-\chi_\rho(z_2, p_2, y)\geq  (\Pi^R(z_2,p_2)-m(\delta)) h_1-M^*,
\end{equation}
where 
$\Pi^R(z_2,p_2)$ is  given by \eqref{eq:special_notation_for_Pi} and $\chi_\rho(z_2,p_2,\cdot)$ is a solution of \eqref{trunc-cellp} given by Lemma \ref{lem:existence_solution_truncated_cell_pb}.

\medskip

 Similarly, let $z_2,p_2\in \R$ be such that  $E(z_2, p_2)>E_0^L(z_2,p_2)$.
There exists $\rho^*>0$, $\delta^*>0$, $m(\cdot):\R_+ \to \R_+$ and $M^*$ as above, such that
 for all  $\delta\in (0,\delta^*]$, $\rho\ge \rho^*$, $y=(y_1,y_2)\in [-\rho,-\rho^*]\times \R$, $h_1\in [0,\rho+y_1]$ and $h_2\in \R$,
\begin{equation}
\label{slope3_bis}
\chi_\rho(z_2, p_2, y-h_1e_1+h_2 e_2)-\chi_\rho(z_2, p_2, y)\geq  -(\Pi^L(z_2,p_2)+m(\delta)) h_1-M^*.
\end{equation}
\end{proposition}
\begin{proof}
Let us focus on \eqref{slope3} since the proof of \eqref{slope3_bis} is similar. Recall that $\rho\mapsto \lambda_\rho(z_2,p_2)$ is nondecreasing and tends to $E(z_2,p_2)$ as $\rho\to+\infty$. Choose $\rho^*=\rho^*(z_2,p_2)>0$ such that for any $\rho\ge \rho^*$, $E(z_2, p_2)>\lambda_\rho(z_2,p_2)>E_0^R(z_2,p_2)$.
 Then, choose $\delta^*=\delta^*(z_2,p_2)>0$ such that for any $\delta\in (0,\delta^*]$, $\lambda_\rho(z_2,p_2)-\delta>E_0^R(z_2,p_2)$.
\begin{figure}[H]
\begin{center}
  \begin{tikzpicture}[scale=0.60]
\draw (-1,1.5) -- (2,1.5);
\draw (2,1.5) ..controls +(2,0) and +(-0.5,-2).. (7,6);
\draw (9.7,6) node[right,above]{$q\mapsto H^{-,R}((0, z_2),p_2e_2+q e_1)$} ;
\draw[dashed] (-1,5.3) -- (8,5.3);
\draw (8,5.3) node[right]{$E(z_2,p_2)$};
\draw[dashed] (2,1.5) -- (8,1.5);
\draw (8,1.5) node[right]{$E_0^R(z_2,p_2)$};
\draw[dashed] (-1,4.2) -- (8,4.2);
\draw (8,4.2) node[right]{$\lambda_{\rho^*}(z_2,p_2)$};
\draw[dashed] (-1,2.2) -- (8,2.2);
\draw (8,2.2) node[right]{$\lambda_{\rho^*}(z_2,p_2)-\delta^*$};
\draw[dashed] (-1,2.9) -- (8,2.9);
\draw (8,2.9) node[right]{$\lambda_{\rho}(z_2,p_2)-\delta$};
\draw[dashed] (6.78,0.5) -- (6.78,5.3);
\draw (6.78,0.5) node[below]{$\Pi^R(z_2,p_2)$};
\draw[dashed] (4.98,0.5) -- (4.98,2.9);
\draw (4.98,0.4) node[below]{$q_\delta$};
\draw[dashed] (4.09,0.5) -- (4.09,2.2);
\draw (4.09,0.4) node[below]{$q_{\delta^*}$};
\end{tikzpicture}
\caption{Construction of $\rho^*,\delta^*$ and $q_\delta$}
    \label{fig:proof_slop_properties}
\end{center}
\end{figure}
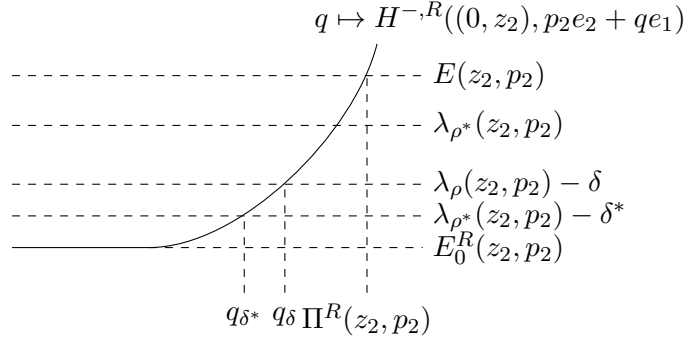
Let us fix $\rho> \rho^*$, $\delta\in (0,\delta^*]$ and $\bar y=(\bar y_1,\bar y_2)\in [\rho^*,\rho]\times \R$. Consider $y \mapsto \chi_\rho(z_2,p_2,y)$ 
a solution of \eqref{trunc-cellp} as in Lemma \ref{lem:existence_solution_truncated_cell_pb}. The function $\chi_\rho(z_2,p_2,\cdot)$ is 1-periodic 
with respect to $y_2$ and Lipschitz continuous with constant $L=L(p_2)$. Thus, for any $y\in \{\bar y_1\}\times \R$ 
\begin{equation*}
\chi_ \rho (z_2,p_2,y)-\chi_\rho (z_2,p_2,\bar y) \geq -L.
\end{equation*}
Therefore, $y\mapsto \chi_\rho(z_2,p_2,y)-\chi_\rho (z_2,p_2,\bar y)$ is a supersolution of 
\begin{equation}
\label{eq:proof_lem_control_slopes_bounded_value_pb_super}
\left\lbrace
\begin{array}{ll}
    \tilde{H}^R((0,z_2),Dv(y)+p_2e_2,y_2)  \ge \lambda_\rho(z_2,p_2), & y\in (\bar y_1,\rho)\times \R,\\
      \tilde{H}^{-,R}((0,z_2),Dv(y)+p_2e_2,y_2)  \ge \lambda_\rho(z_2,p_2), & y\in \{\rho\}\times \R,\\
      v(y)  \ge -L, & y\in \{\bar y_1\}\times \R,\\
      v \hbox{ is 1-periodic w.r.t. } y_2.
\end{array}
\right.
\end{equation}
Since $\rho\ge \rho^*$  and $\delta\in (0,\delta^*]$, there exists a unique $q_\delta\in \R$, see Figure \ref{fig:proof_slop_properties}, such that 
\begin{equation}
\label{slope1}
\lambda_\rho(z_2, p_2)-\delta=H^R((0, z_2),p_2e_2+q_\delta e_1)=H^{-,R}((0, z_2),p_2e_2+q_\delta e_1).
\end{equation}
We observe that $q_{\delta^*}\le q_\delta \le \Pi^R(z_2,p_2)$ and that $q_\delta$ tends to $\Pi^R(z_2,p_2)$ as $\delta$ tends to $0$.
 Choose $ m(\delta)=\Pi^R(z_2,p_2)-q_\delta\ge 0$ and consider the function $w^R$: $w^R(y)=q_\delta(y_1+g(y_2))$, which is of class  $\cC^2$. 
From the choice of $q_\delta$, for any $y\in \R^2$,
\begin{equation}
\label{eq:proof_lem_control_slopes_relation_satisfies_by_w^2}
\tilde{H}^{-,R}(z_2,Dw^R(y)+p_2e_2,y_2)\le \tilde{H}^R(z_2,Dw^R(y)+p_2e_2,y_2)=\lambda_\rho(z_2, p_2)-\delta,
\end{equation}
and for any $y\in \{\bar y_1\}\times \R$,
\begin{displaymath}
w^R(y)-q_\delta \bar y_1\le |q_\delta|\parallel g\parallel_\infty\le \max\left( |q_{\delta^*}|, |\Pi^R(z_2,p_2)| \right) \parallel g\parallel_\infty=C= C(z_2,p_2).
\end{displaymath}
Therefore, as stated in Remark~\ref{sec:furth-results-corr-1}, the subsolution property holds up to the boundary $y_1=\rho$ and 
 the function  $u^{R}:  y  \in  [\bar y_1,R]\times\R \mapsto w^R(y)-q_\delta \bar y_1- C-L$  is a subsolution of
\begin{equation}
\label{eq:proof_lem_control_slopes_bounded_value_pb_sub}
\left\lbrace
\begin{array}{ll}
    \tilde{H}^R((0,z_2),Du(y)+p_2e_2,y_2)  \le \lambda_\rho(z_2,p_2)-\delta & y\in (\bar y_1,\rho)\times \R,\\
      \tilde{H}^{-,R}((0,z_2),Du(y)+p_2e_2,y_2)  \le \lambda_\rho(z_2,p_2)-\delta, & y\in \{\rho\}\times \R,\\
      u(y)  \le -L, & y\in \{\bar y_1\}\times \R,\\
      u \hbox{ is 1-periodic w.r.t. } y_2.
\end{array}
\right.
\end{equation}
Finally, since $\chi_\rho(z_2,p_2,y)-\chi_\rho (z_2,p_2,\bar y)$ is a supersolution of \eqref{eq:proof_lem_control_slopes_bounded_value_pb_super} 
and $u^R$ is a subsolution of \eqref{eq:proof_lem_control_slopes_bounded_value_pb_sub},
 the comparison principle in Lemma~\ref{sec:furth-prop-corr} yields:  for all $y\in [\bar y_1,\rho]\times \R$
\begin{equation}
\label{eq:proof_lem_control_slopes_bounded_value_pb_end}
\begin{array}{lll}
\chi_\rho(z_2,p_2,y)-\chi_\rho (z_2,p_2,\bar y) \geq u^{R}(y) & = & q_\delta((y_1-\bar y_1)+g(y_2))- C-L\\
& \geq & (\Pi^R(z_2,p_2)-m(\delta)) ( y_1-\bar y_1)-M^*,
\end{array}
\end{equation}
where $M^*$ is a constant depending only of $z_2$ and $p_2$. 
Note that the constants which appear in \eqref{eq:proof_lem_control_slopes_bounded_value_pb_end} are independent of $\rho>0$. 
\end{proof}
The following corollary deals with the global corrector $\chi$:
\begin{corollary}
\label{cor:slopes_omega}
With  $\Pi^i(z_2,p_2)$  defined in \eqref{eq:special_notation_for_Pi}, $i=L,R$,
\begin{enumerate}
\item If $E(z_2, p_2)>E_0^R(z_2,p_2)$, then, with   $\rho^*>0$ and $M^*\in \R$ as in  the first point of  Proposition \ref{SLOPESLemma1},
for all $y\in [\rho^*,+\infty)\times \R$, $h_1\ge 0$ and $h_2\in \R$,
\begin{equation}
\label{slope32}
\chi(z_2, p_2, y+h_1e_1+h_2 e_2)-\chi(z_2, p_2, y)\geq  \Pi^R(z_2,p_2) h_1-M^*.
\end{equation}
\item  If $E(z_2, p_2)>E_0^L(z_2,p_2)$, then, with   $\rho^*>0$ and $M^*\in \R$ as in  the second point of  Proposition \ref{SLOPESLemma1},
for all  $y\in (-\infty,-\rho^*]\times \R$, $h_1\ge 0$ and $h_2\in \R$,
\begin{equation}
\label{slope3_bis2}
\chi(z_2, p_2, y-h_1e_1+h_2 e_2)-\chi(z_2, p_2, y)\geq - \Pi^L(z_2,p_2) h_1-M^*.
\end{equation}
\end{enumerate}
\end{corollary}
\begin{proof}
The proof follows easily from Proposition \ref{SLOPESLemma1} and the local uniform 
convergence of the sequence $\chi_\rho(z_2, p_2,\cdot)$ toward $\chi(z_2, p_2,\cdot)$.
\end{proof}

\begin{corollary}
\label{cor:control_slopes_W}
For $z_2,p_2\in \R$,  $y\mapsto W(z_2,p_2,y)$ defined in Lemma~\ref{lem:rescaling_omega} satisfies 
\begin{eqnarray}
\label{cor:control_slopes_W1}
\overline{\Pi}^R(z_2,p_2)\le \partial_{y_1}W(z_2,p_2,y) \le \widehat{\Pi}^R(z_2,p_2) & &\hbox{ for a.a. } y\in (0,+\infty)\times \R,\\
\label{cor:control_slopes_W2}
 \widehat{\Pi}^L(z_2,p_2) \le \partial_{y_1}W(z_2,p_2,y) \le \overline{\Pi}^L(z_2,p_2) & &\hbox{ for a.a. }y\in (-\infty,0)\times\R,
\end{eqnarray}
and for all $y$:
\begin{equation}
\label{eq:control_slopes_W_summary}
\overline{\Pi}^L(z_2,p_2) y_1 1_{\Omega^L}+\overline{\Pi}^R(z_2,p_2) y_1 1_{\Omega^R}\le W(z_2,p_2,y) \le  \widehat{\Pi}^L(z_2,p_2) y_1 1_{\Omega^L}+ \widehat{\Pi}^R(z_2,p_2) y_1 1_{\Omega^R}.
\end{equation}
\end{corollary}
\begin{proof}
 From Lemma~\ref{lem:rescaling_omega}, we see that $y\mapsto W(z_2,p_2,y)$ is Lipschitz continuous w.r.t. $y_1$ and independent of $y_2$, and satisfies 
 \begin{equation}
   \label{eq:11}
H^R((0,z_2),\partial_{y_1} W(z_2,p_2,y) e_1+p_2e_2)=E(z_2, p_2) \quad \hbox{  a.e. in } \Omega^R. 
 \end{equation}
Consider first the case when $E(z_2,p_2)>E_0^R(z_2,p_2)$;
from the convexity and coercivity of $H^R$, the observations above yield that almost everywhere in $y$, $\partial_{y_1}W(z_2,p_2,y)$ can be either $\Pi^R(z_2,p_2)$
(the unique real number such that $H^{-,R}((0,z_2), qe_1+p_2e_2)=E(z_2, p_2)$), or the unique real number  $q$ (depending on $(z_2,p_2)$) such that $ H^{+,R}((0,z_2), qe_1+p_2e_2)=E(z_2, p_2)$.  Note that $q<  \Pi^R(z_2,p_2)$.
But from Corollary \ref{cor:slopes_omega} and the local uniform convergence of $W_\epsilon(z_2,p_2,\cdot)$ toward $W(z_2,p_2,\cdot)$, 
we see that  that for any $y_1>0$ and $h_1\ge 0$,
\begin{displaymath}
W(z_2,p_2,y+h_1 e_1)-W(z_2,p_2,y)\ge \Pi^R(z_2,p_2) h_1,
\end{displaymath}
which implies that almost everywhere, $\partial_{y_1}W(z_2,p_2,y) \ge  \Pi^R(z_2,p_2)>q$. Therefore,  
 $\partial_{y_1}W(z_2,p_2,\cdot)= \Pi^R(z_2,p_2)$ almost everywhere.
\\
In the case when  $E(z_2,p_2)=E_0^R(z_2,p_2)$, we deduce from (\ref{eq:11}) that
 almost everywhere in $y$,   $\overline{\Pi}^R(z_2,p_2)\le \partial_{y_1}W(z_2,p_2,y) \le \widehat{\Pi}^R(z_2,p_2)$. 
\\
We have proved (\ref{cor:control_slopes_W1}). The proof of (\ref{cor:control_slopes_W2}) is identical. Finally, (\ref{eq:control_slopes_W_summary}) comes from (\ref{cor:control_slopes_W1}),~(\ref{cor:control_slopes_W2})  and from the fact that $W(z_2,p_2,0)=0$.
\end{proof}

 
\section{A comparison principle for (\ref{def:HJeffective_short})}
\label{sec:comp-princ-refd}
To prove the main result of the paper, i.e. Theorem \ref{th:convergence_result}, we  need a comparison principle for (\ref{def:HJeffective1})-(\ref{def:HJeffective2}).
Before proving such a result, we need to establish some useful properties of $E$ arising in (\ref{def:HJeffective2}).
 \subsection{Properties of $E(\cdot,\cdot)$}
 \label{sec:some_properties_on_E}
 In the  theory of homogenization of Hamilton-Jacobi equations, it is quite standard to observe that the effective Hamiltonian inherits some properties from the original problem, see the pioneering work \cite{LPV}. 
\begin{lemma}
\label{lem:regularity_of_ergodic_hamiltonian}
For any $z_2\in \R$, the function $p_2\mapsto E(z_2,p_2)$ is convex.
 For any $z_2,z_2',p_2, p_2'\in \R$,
 \begin{eqnarray}
   \label{lem:E_lipschitz_wrt_p_2}
\mid E(z_2,p_2)-E(z_2,p_2') \mid \le M_f |p_2-p_2'|,
\\
\label{lem:E_lipschitz_wrt_z_2}
\mid E(z_2,p_2)-E(z_2',p_2) \mid \le C(1+|p_2|) |z_2-z_2'|+\omega(|z_2-z_2'|),
\\
 \label{lem:E_coercive}
 \delta_0|p_2|-M_\ell\le E(z_2,p_2)\le M_f |p_2|+M_\ell,
\end{eqnarray}
where the constants $M_f$, $M_\ell$, $\delta_0$ have been introduced in Assumptions $[\rm{H}0]$-$[\rm{H}3]$, 
the modulus of continuity $\omega$ has been introduced in Lemma \ref{sec:fast-slow-variables-2} and $C$ is a positive constant.
\\
Moreover, $p_2 \mapsto E(z_2,p_2)$ is affine in a neighborhood of $\pm \infty$. More precisely,
for any $z_2\in \R$, there exist $\hat \ell(z_2), \check \ell (z_2)\in [-M_\ell,M_\ell]$, $\hat f (z_2), \check f (z_2)\in [\delta_0,M_f]$ and $\hat K(z_2), \check K(z_2)>0$ such 
\begin{eqnarray}
  \label{eq:E_linear_for_p_big_enough}
p_2\ge \hat K(z_2) &\Rightarrow & E(z_2,p_2)= \hat f (z_2)p_2 +\hat \ell(z_2),\\ 
\label{eq:15}
p_2\le -\check K(z_2) &\Rightarrow & E(z_2,p_2)= -\check f(z_2)p_2 +\check \ell(z_2).
\end{eqnarray}
\end{lemma}

\begin{proof}
The proof contains arguments that are  quite similar to those contained in \cite{LPV}, but technical difficulties arise from the discontinuities of the Hamiltonians at $y_1=0$.
The main idea is to deduce the desired properties from those of $-\eta v_\rho^\eta$, where $v_\rho^\eta$ is defined in~(\ref{valueR}).
For brevity, we only  prove    (\ref{lem:E_lipschitz_wrt_z_2}) and that   $p_2 \mapsto E(z_2,p_2)$ is affine in a neighborhood of $\pm \infty$.

\paragraph{Proof of (\ref{lem:E_lipschitz_wrt_z_2})}
For $z_2,z_2',p_2\in \R$, consider $y\mapsto v^\eta_\rho(z_2, p_2,y)$ and $y\mapsto v^\eta_\rho(z_2', p_2,y)$ given by \eqref{valueR}. 
These functions are  viscosity solutions of \eqref{trunc-cellpapprox}   with $\lambda_\eta=0$.  
Assume that 
$0 = v^\eta_\rho(z_2', p_2,\bar y) -\varphi(\bar y)$ be a local minimum of $v^\eta_\rho(z_2', p_2,\cdot) -\varphi(\cdot)$ for  $\bar y\in \R^2$ and $\varphi \in \cR_\rho$.
As above, we focus on the case when $\bar y \in \Gamma$ because the other cases are simpler.
It is not restrictive to assume that $\varphi^L$   and   $\varphi^R$  are smooth (at least $\cC^3$).
Since $y\mapsto v^\eta_\rho(z_2', p_2,y)$ is Lipschitz continuous with a constant $ L(p_2)=  L_1 + L_2 |p_2|$, see (\ref{eq:v_R^rho_L_lipschitz}), we see that 
\begin{equation}
 \label{proof:point_(ii)_regularity_E_1}
   |\partial_{y_2} \varphi^L(\bar y)|= |\partial_{y_2} \varphi^R(\bar y)| \le L(p_2),\quad 
  \partial_{y_1} \varphi^L(\bar y)\ge -L(p_2),\quad 
  \partial_{y_1} \varphi^R(\bar y)\le L(p_2).
\end{equation}
It is always possible to modify $\varphi$ and obtain a test-function $\psi$ such that $|\partial_{y_1} \psi^i(\bar y)|\le 2L(p_2)$, 
$|\partial_{y_2} \psi^i(\bar y)|\le L(p_2)$,
$i=L,R$ and that
$v^\eta_\rho(z_2', p_2,\cdot)-\psi(\cdot)$  has a local minimum at $\bar y$. Indeed, we make out two cases:
\begin{enumerate}
\item if $|\partial_{y_1} \varphi^i(\bar y)|\le 2L(p_2)$ for $i=L,R$,  then we choose $\psi=\varphi$.

\item If $\partial_{y_1} \varphi^L(\bar y)>2L(p_2)$ or $\partial_{y_1} \varphi^R(\bar y)<-2L(p_2)$, let us introduce  
\begin{displaymath}
\!\!\!\!\!\!\!\!\!\!\!\!\!\!\!\!\!\!\!\!\!\!\!\!\!\!\!\!\!\!   \psi(y) =\left\lbrace
\begin{array}{ll}
 \varphi(y)-(2L(p_2)+\partial_{y_1}\varphi^R(\bar y))y_1-A  |y-\bar y|^2  & \mbox{ if } y\in [0,\rho]\times \R \mbox{ and } \partial_{y_1} \varphi^R(\bar y)<-2L(p_2),\\
  \varphi(y)-A |y-\bar y|^2  & \mbox{ if } y\in [0,\rho]\times \R \mbox{ and } |\partial_{y_1} \varphi^R(\bar y)|\le 2L(p_2),\\
\varphi(y)+(2L(p_2)-\partial_{y_1}\varphi^L(\bar y))y_1-A|y-\bar y|^2 &  \mbox{ if } y\in [-\rho,0]\times \R  \mbox{ and } \partial_{y_1} \varphi^L(\bar y)>2L(p_2),\\
\varphi(y)-A |y-\bar y|^2 &  \mbox{ if } y\in [-\rho,0]\times \R  \mbox{ and } |\partial_{y_1} \varphi^L(\bar y)|\le 2L(p_2).
\end{array}
\right.
\end{displaymath}
Note that $|\partial_{y_1} \psi^i(\bar y)|\le 2L(p_2)$, $i=L,R$  and that 
\begin{equation}
  \label{eq:14}
\partial_{y_1} \varphi^L(\bar y)\ge \partial_{y_1} \psi^L(\bar y),\quad\hbox{and}\quad
\partial_{y_1} \varphi^R(\bar y)\le \partial_{y_1} \psi^R(\bar y).
\end{equation}
We claim that for $A$ large enough,  $v^\eta_\rho(z_2', p_2,\cdot)-\psi(\cdot)$  has a local minimum at $\bar y$. Indeed, fix $r>0$ such that 
$v^\eta_\rho(z_2', p_2,y)-\varphi(y)\ge v^\eta_\rho(z_2', p_2,\bar y)-\varphi(\bar y)=0$  for  $y\in B(\bar y, r)$.
Assuming for example that $\partial_{y_1} \varphi^R(\bar y)<-2L(p_2)$ (the case when $ |\partial_{y_1}\varphi^R(\bar y)|\le 2L(p_2)$ is obvious),
 we see that  for $y\in B(\bar y, r)$ with $y_1>0$,
 \begin{equation}
   \label{eq:16}
   \begin{split}
v^\eta_\rho(z_2', p_2,y)-\psi(y)  \ge& v^\eta_\rho(z_2', p_2,y)- v^\eta_\rho(z_2', p_2,(0,y_2))+  \varphi(0,y_2)-\varphi(y)\\ &+ \left(2L(p_2)+\partial_{y_1}\varphi^R(\bar y)\right)y_1+A|y-\bar y|^2.     
   \end{split}
 \end{equation}
For a constant $c>0$ depending on $\phi$ and $r$, 
\begin{eqnarray*}
\varphi(0,y_2)-\varphi(y) &\ge&  -y_1 \partial_{y_1} \varphi^R(0,y_2) -  c y_1^2  \ge  - y_1\partial_{y_1} \varphi^R (\bar y) -  c ( y_1^2 +y_1 |y_2-\bar y_2|).
\end{eqnarray*}
On the other hand, $ v^\eta_\rho(z_2', p_2,y)- v^\eta_\rho(z_2', p_2,(0,y_2))\ge  -L(p_2)y_1$. From the latter  two observations and   (\ref{eq:16}),
we deduce that $v^\eta_\rho(z_2', p_2,y)-\psi(y) \ge  L(p_2)y_1 - c ( y_1^2 +y_1 |y_2-\bar y_2|) + A |y-\bar y|^2$. Therefore for $A$ large enough, 
\begin{displaymath}
v^\eta_\rho(z_2', p_2,y)-\psi(y) \ge  0 =v^\eta_\rho(z_2', p_2,\bar y)-\psi(\bar y),\quad y\in   B(\bar y, r),
\end{displaymath}
and the claim is proved.
\end{enumerate}
In both cases, we see  that $ \eta v^\eta_\rho(z_2', p_2,\bar y)+   \max_{i\in \{L,R\}}\left(\tilde H^{+,i}((0,z_2'),D\psi^i(\bar y)+p_2e_2,\bar y_2)\right) \ge 0$; 
assuming  that the
latter maximum is achieved by $i=R$ for example, this yields
\begin{displaymath}
\begin{array}{l}
 \eta\left( v^\eta_\rho(z_2', p_2,\bar y)+\frac{C}{\eta}(1+|p_2|) |z_2-z_2'|+\frac{1}{\eta}\omega(|z_2-z_2'|)\right)+   
\tilde H^{+,R}((0,z_2),D\psi^R(\bar y)+p_2e_2,\bar y_2) \\
\ge C(1+|p_2|) |z_2-z_2'|-M|p_2e_2+D\psi^R(\bar y)||z_2-z_2'|.
 \end{array}
\end{displaymath}
Thanks to \eqref{proof:point_(ii)_regularity_E_1} and from the construction of $\psi$,
  \begin{equation}
  \label{proof:point_(ii)_regularity_E_3}
 |p_2e_2+D\psi^i(\bar y)|\le |p_2|+3L(p_2), \quad i=L,R.
 \end{equation}
 Hence, choosing $C\ge M \max( 3 L_1, 1+ 3L_2)$ yields 
\begin{displaymath}
\begin{array}{l}
 \eta\left( v^\eta_\rho(z_2', p_2,\bar y)+\frac{C}{\eta}(1+|p_2|) |z_2-z_2'|+\frac{1}{\eta}\omega(|z_2-z_2'|)\right)+  
\tilde H^{+,R}((0,z_2),D\psi^R(\bar y)+p_2e_2,\bar y_2) 
 \ge 0 .
 \end{array}
\end{displaymath}
But from (\ref{eq:14}) and the nonincreasing character of $\tilde H^{+,R}$, we see that 
\begin{displaymath}
\begin{array}{l}
 \eta\left( v^\eta_\rho(z_2', p_2,\bar y)+\frac{C}{\eta}(1+|p_2|) |z_2-z_2'|+\frac{1}{\eta}\omega(|z_2-z_2'|)\right)+  
\tilde H^{+,R}((0,z_2),D\varphi^R(\bar y)+p_2e_2,\bar y_2) 
 \ge 0. 
 \end{array}
\end{displaymath}
Therefore, $v^\eta_\rho(z_2', p_2,\cdot)+\frac{C}{\eta}(1+|p_2|) |z_2-z_2'|+\frac{1}{\eta}\omega(|z_2-z_2'|)$ is a supersolution of the equation satisfied by $v^\eta_\rho(z_2, p_2,\cdot)$ and we conclude using the comparison principle Lemma \ref{th:comparison_truncated_cell_pb}, passing to the limit as $\eta\to 0$ and $\rho\to +\infty$, and finally exchanging the roles of $z_2$ and $z_2'$.
\paragraph{Proof that $p_2 \mapsto E(z_2,p_2)$ is affine in a neighborhood of $\pm \infty$}
We focus on \eqref{eq:E_linear_for_p_big_enough} since the proof of  (\ref{eq:15}) is similar. For $z_2\in \R$, $y\in [-\rho,\rho]\times \R$ and $p_2,\eta,\rho>0$,
let us define  $\bar f_\rho^\eta(z_2,y)=\sup_{(\gamma_y,a)\in\cT_{z_2,y,\rho}} \left\lbrace-\int_0^\infty f_2((0, z_2),a(t))e^{-\eta t} dt\right\rbrace$, with 
$\cT_{z_2,y,\rho}$ given in (\ref{eq:2freezed}). From 
\eqref{def:ell_pb_oc_truncated} and \eqref{valueR}, we deduce that 
\begin{equation}
\label{eq:proof_E_linear_wrt_p_1}
\eta \bar f_\rho^\eta(z_2,y)  p_2 -M_\ell  \le  -\eta v^\eta_\rho(z_2, p_2,y) 
\le  \eta \bar f_\rho^\eta(z_2,y)  p_2 +M_\ell.
\end{equation}
From the assumptions, it is easy to check that 
\begin{eqnarray}
\label{eq:proof_E_linear_wrt_p_2}
|\eta  \bar f_\rho^\eta(z_2,y)| & \le &  M_f,\\
\label{eq:proof_E_linear_wrt_p_3}
| \bar f_\rho^\eta(z_2,y)- \bar f_\rho^\eta(z_2,y')| & \le &  C|y-y'|,
\end{eqnarray}
for some positive constant $C$, 
and that $y\mapsto \bar f_\rho^\eta(z_2,y)$ is periodic with period $1$ in the variable $y_2$. From Ascoli-Arzela theorem, up to the extraction of subsequences, we may assume that 
$\bar f_\rho^\eta(z_2,\cdot)- \bar f_\rho^\eta(z_2,0)$ tends to a Lipschitz function (with Lipschitz constant $C$) and that $\eta \bar f_\rho^\eta(z_2,\cdot)$ tends to a constant $\bar f_\rho (z_2)$  as $\eta\to 0$. With the same arguments as in \S~\ref{sec:passage-limit-as}, we may prove that $\bar f_\rho (z_2)$  is nondecreasing and  uniformly bounded with respect to $\rho$. Therefore, we may define $\hat f (z_2)=\lim_{\rho\to +\infty} \bar f_\rho (z_2)$.\\
Passing to the limit in (\ref{eq:proof_E_linear_wrt_p_1}) as $\eta\to 0$ then as  $\rho\to +\infty$, we deduce that
\begin{equation}
\label{eq:proof_E_linear_wrt_p_end}
\hat f  (z_2) p_2 -M_\ell  \le  E(z_2, p_2)   \le \hat f (z_2) p_2 +M_\ell.
\end{equation}
Finally, from (\ref{eq:proof_E_linear_wrt_p_end}) and the convexity of $p_2\mapsto E(z_2, p_2) $, we infer that 
there exists $\hat \ell(z_2)\in [-M_\ell,M_\ell]$ and $\hat K(z_2)>0$ such that for any $p_2\ge \hat K(z_2)$
\begin{displaymath}
 E(z_2, p_2) = \hat f (z_2) p_2 + \hat \ell (z_2).
\end{displaymath}
Finally, the bound $\hat f (z_2)\ge \delta_0$ comes from (\ref{lem:E_coercive}), and it is simple to check that $\hat f (z_2)\le M_f$.
\end{proof}

\subsection{The comparison principles}
\label{sec:comp-princ}
Since we are not able to control the constants $\hat K(z_2)$ and $\check K (z_2)$ arising in Lemma \ref{lem:regularity_of_ergodic_hamiltonian},
 we cannot directly use the comparison principle which is available in \cite[Theorem 2.5]{oudet2014}.   To apply the latter,
 it will be useful to first modify $E(z_2,p_2)$ for $|p_2|$ larger than some fixed  number $K$ independent of $z_2$. The following lemma deals with such modified Hamiltonians.
\begin{lemma}
\label{sec:proof-that-p_2}
  For a positive number $K$,  the Hamiltonian $E_K(z_2, p_2)$ defined by 
\begin{equation}
 \label{eq:def_E_K}
 E_K(z_2,p_2)=\left\lbrace
 \begin{array}{rcl}
 E(z_2,p_2) & \hbox{ if }& |p_2|\le K,\\
 E(z_2,K) +M_f(p_2-K) & \hbox{ if }& p_2> K,\\
  E(z_2,-K) -M_f(p_2+K) & \hbox{ if }& p_2<- K,
 \end{array}
 \right.
 \end{equation}
is convex in the variable $p_2$ and 
\begin{equation}
 \label{eq:representation_of_E_K}
 E_K(z_2,p_2)=\max_{b\in [-M_f,M_f]}\left(bp_2-E_K^\star(z_2,b)  \right),
 \end{equation}
where $E_K^\star: \R^2 \to \R \cup \{+\infty\}$ is the  Fenchel transform $ E_K^\star(z_2,b)=\sup_{q\in \R}\left(bq-E_K(z_2,q)\right)$.
 For  $z_2,z_2'\in \R$ and $b,b'\in [-M_f,M_f]$,
\begin{eqnarray}
\label{eq:E_K_star_bounded}
 |E_K^\star(z_2,b)|\le C_K,\\
 \label{prop:regularity_E_star_K_1}
 | E_K^\star(z_2,b)- E_K^\star(z_2,b')|\le K|b-b'|,  \\
 \label{prop:regularity_E_star_K_2}
 | E_K^\star(z_2,b)- E_K^\star(z_2',b)|\le \omega_K(t)= C(1+K)|z_2-z_2'|+\omega(|z_2-z_2'|),\end{eqnarray}
where in \eqref{eq:E_K_star_bounded}, $C_K$ is a positive constant, and,  in \eqref{prop:regularity_E_star_K_2}, the constant $C$ and the modulus of continuity $\omega$ are those appearing in (\ref{lem:E_lipschitz_wrt_z_2}).
\end{lemma}
\begin{proof}
The convexity of $p_2\mapsto E_K(z_2,p_2)$ comes from the convexity of $p_2\mapsto E(z_2,p_2)$ and from  (\ref{lem:E_lipschitz_wrt_p_2}).
From (\ref{lem:E_lipschitz_wrt_p_2}), it is also clear that $E_K^\star(z_2,b)=+\infty$ if $b\notin [-M_f, M_f]$, which implies (\ref{eq:representation_of_E_K}).
It can also be seen that if $b\in [-M_f, M_f]$, then 
\begin{equation}
\label{eq:sup_reached_E_K}
 E_K^\star(z_2,b)=\max_{p\in [-K,K]}\left(bp-E_K(z_2,p)\right).
\end{equation}
From (\ref{lem:E_coercive}), we check that  for $z_2\in \R$, $p_2\in [-K,K]$ and $b\in [-M_f,M_f]$,
\begin{equation}
\label{eq:proof_E_K_bounded}
bp_2 -M_f|p_2|-M_\ell \le bp_2- E_K(z_2,p_2) \le bp_2+ M_f(K-|p_2|)-\delta_0 K+M_\ell.
\end{equation}
Choosing $p_2=0$ yields that $E_K^\star(z_2,b_2)\ge -M_\ell$. Using the fact that $bp_2- M_f|p_2|\le 0$ if $|b|\le M_f$, we deduce that 
$E_K^\star(z_2,b)\le   C_K$  and finally (\ref{eq:E_K_star_bounded}) with   $ C_K =(M_f-\delta_0)K+M_\ell$. \\
It is standard to deduce \eqref{prop:regularity_E_star_K_1} and \eqref{prop:regularity_E_star_K_2} from  \eqref{eq:sup_reached_E_K}
and (\ref{lem:E_lipschitz_wrt_z_2}).
\end{proof}
Lemma~\ref{sec:proof-that-p_2} allows us to prove  the following comparison principle:
\begin{proposition}
\label{thm:comparison_principle_effective_eq_with_E_K}
 Let $u$  and $w$ be respectively a bounded subsolution and a bounded supersolution of
 \begin{equation}
\label{eq:effective_eq_with_E_K}
\begin{array}{lll}
\lambda u(z)+ H^i(z, D u(z))  = 0 & &\hbox{if } z\in \Omega^i,\\
\lambda u(z)+\max\left(E_K(z_2,\partial_{z_2}\varphi(z)),H_\Ga(z, D u^L(z),D u^R(z))\right)= 0  & &\hbox{if } z=(0,z_2)\in \Ga.
\end{array}
\end{equation} 
Then, $u\le w$ in $\R^2$.
\end{proposition}
\begin{proof}
Thanks to Lemma~\ref{sec:proof-that-p_2},  it is possible to apply  \cite[Theorem 2.5]{oudet2014}, more precisely
the general comparison principle discussed in \cite[Remark~2.11]{oudet2014}, since  the continuity of $u$ is not assumed.
\end{proof}
\begin{theorem}
\label{th:comparison_effective_HJB_eq}
Let $u:\R^2\to \R$ and   $v:\R^2\to \R$  be respectively a bounded subsolution and  a bounded supersolution of \eqref{def:HJeffective_short}. Then $u\le v$ in $\R^2$.
\end{theorem}
\begin{proof}
Since $u$ is a subsolution of \eqref{def:HJeffective_short}, it is also a subsolution of
\begin{eqnarray*}
\lambda u(z)+ H^i(z, D u(z))  = 0 & &\hbox{if } z\in \Omega^i,\\
\lambda u(z)+H_\Ga(z, D u^L(z),D u^R(z))= 0  & &\hbox{if } z=(0,z_2)\in \Ga.  
\end{eqnarray*}
Thanks to Assumptions $[\rm{H}0]$-$[\rm{H}3]$, we can apply \cite[Lemma~2.6]{oudet2014} to $u$:
 there exists $r>0$ such that $u$ is Lipschitz continuous in $ [-r,r]\times \R$.
 Let us call $L_u$ the Lipschitz constant of the restriction of $u$ to $ [-r,r]\times \R$  and choose $K\ge L_u$. Since $E_K$ coincides with $E$ on $\R\times [-K,K]$, we deduce that  $u$ is a subsolution of \eqref{eq:effective_eq_with_E_K}. \\
On the other hand, since $E_K\ge E$,  $v$ is a supersolution of \eqref{eq:effective_eq_with_E_K}.\\
The proof is achieved by applying Proposition \ref{thm:comparison_principle_effective_eq_with_E_K} to the pair $(u,v)$.
\end{proof}

\section{Proof of the main result}\label{sec:proof-main-result}
\subsection{ A reduced set of test-functions}
From \cite{imbert:hal-00832545} and \cite{imbert:hal-01073954},  we may use an equivalent definition for the  viscosity  solution of (\ref{def:HJeffective_short}).
\begin{definition}
\label{def:test_functions_set_restricted}
Recall that $\overline{\Pi}^i$ and $\widehat{\Pi}^i$, $i=L,R$, have been introduced in \eqref{eq:5}-\eqref{eq:10}.
Let $\Pi: \R^2 \to \R^2$,  $(z_2, p_2)  \mapsto \left( \Pi^L(z_2,p_2),\Pi^R(z_2,p_2) \right)$ be such that, for all $(z_2,p_2)$
\begin{equation}
  \label{eq:18}
  \begin{split}
    \widehat{\Pi}^L(z_2,p_2)  &\le   \Pi^L (z_2,p_2)  \le  \overline{\Pi}^L(z_2,p_2) ,\\ 
\overline{\Pi}^R (z_2,p_2)   &\le   \Pi^R(z_2,p_2)   \le  \widehat{\Pi}^R(z_2,p_2) .
  \end{split}
\end{equation}
For $\bar z=(0,\bar z_2) \in \Gamma$, the reduced set of test-functions  $\cR^\Pi(\bar z)$ associated to the map $\Pi$ is
the set of the functions  $\varphi\in \cC^0 (\R^2)$ such that there exists a $\cC^1$ function $\psi:\R\to \R$ with 
\begin{equation}
  \label{eq:17}
\varphi(z)= \psi(z_2)+   \left( \Pi^R\left(\bar z_2, \psi'(\bar z_2) \right) 1_{z\in \Omega^R}+\Pi^L \left(\bar z_2, \psi'(\bar z_2)\right) 1_{z\in \Omega^L} \right) z_ 1.
\end{equation}
\end{definition}
The following theorem is reminiscent of \cite[Theorem~2.6]{imbert:hal-00832545}.
\begin{theorem} \label{th:restriction_set_of_test_functions}
 Let $u:\R^2\to \R$ be a subsolution (resp. supersolution) of \eqref{def:HJeffective1} and a map $\Pi: \R^2\to \R^2$, $(z_2, p_2)  \mapsto \left( \Pi^L(z_2,p_2),\Pi^R(z_2,p_2) \right)$ such that \eqref{eq:18} holds for all $(z_2,p_2)\in \R^2$.
\\
The function $u$ is a subsolution (resp. supersolution) of \eqref{def:HJeffective2}
 if and only if for any $z=(0,z_2)\in \Ga$ and for all $\varphi \in \cR^\Pi(z)$ such that $u-\varphi$ has a local maximum (resp. local minimum) at $z$,
\begin{equation}
\label{eq:th_restriction_set_of_test_functions}
\lambda u(z)+\max\left(E(z_2,\partial_{z_2}\varphi(z)), H_\Ga(z, D \varphi^L(z),D \varphi^R(z)) \right) \le 0, \quad (\hbox{resp. }\ge 0),
\end{equation}
where the meaning of $D \varphi^L$ and $D \varphi^R$  is given in Definition~\ref{adtest}.
\end{theorem}
\begin{proof}
The proof  follows the lines of that of \cite[Theorem~2.6]{imbert:hal-00832545}
 and is given in Appendix \ref{appendix:proof_th_restriction_set_of_test_functions} for the reader's convenience.
It is worth to note that Lemma \ref{sec:effect-hamilt-gamma-1} is  important in order to use the arguments 
contained in the proof of \cite[Theorem~2.6]{imbert:hal-00832545}.
\end{proof}

\subsection{Proof of Theorem~\ref{th:convergence_result}}

As seen in  $\S$ \ref{sec:asymptotic_behavior}, the result will be proved if we show that the sequence $\left(\tilde v_\epsilon\right)_\epsilon$  corresponding to the straightened geometry
converges to $v$.  We will actually prove that $\overline{\tilde{v}}$  and $\underline{\tilde{v}}$ defined in \eqref{def:v_tilde_overline_underline} are respectively a subsolution and a supersolution of \eqref{def:HJeffective_short}. From  Theorem \ref{th:comparison_effective_HJB_eq}, this will imply that  $\overline{\tilde{v}}=\underline{\tilde{v}}=v=\lim_{\epsilon\to 0}\tilde v_\epsilon$. Moreover, from Proposition \ref{supersub}, we just have to check the transmission condition \eqref{def:HJeffective2}.

\medskip

We restrict ourselves to checking that $\overline{\tilde{v}}$ is a subsolution of \eqref{def:HJeffective_short}, since the proof that $\underline{\tilde{v}}$ is a supersolution of \eqref{def:HJeffective_short} is similar. Take $\bar z=(0,\bar z_2)\in \Ga$. 
We are going to use Theorem \ref{th:restriction_set_of_test_functions} with the special choice for the map $\Pi:  \R^2 \to \R^2$: 
$\Pi(z_2,p_2)= \left( \overline \Pi^L   (z_2,p_2) ,\overline \Pi^R   (z_2,p_2)       \right)$. 
Take a  test-function $\varphi\in \cR^\Pi(\bar z)$, i.e. of the form 
\begin{equation}
\label{eq:19}
\varphi(z)= \psi(z_2)+   \left( \overline \Pi^R   (\bar z_2,   \psi'(\bar z_2) ) 1_{z\in \Omega^R}+\overline \Pi^L(\bar z_2,   \psi'(\bar z_2)) 1_{z\in \Omega^L} \right) z_ 1,
\end{equation}
for  a $\cC^1$ function $\psi:\R\to \R$,  such that $\overline{\tilde{v}}-\varphi$ has a strict local maximum at $\bar z$ and that $\overline{\tilde{v}}(\bar z)=\varphi(\bar z)$. \\
 Let us proceed by contradiction and assume that 
\begin{equation}
\label{eq:proof_convergence_sub_contradiction}
\lambda \varphi(\bar z)+ \max\left(E(\bar z_2,\partial_{z_2}\varphi(\bar z)), H_\Ga(\bar z, D \varphi^L(\bar z),D \varphi^R(\bar z)) \right)=\theta >0.
\end{equation}
From  (\ref{eq:19}),  we see that $H_\Ga(\bar z, D \varphi^L(\bar z),D \varphi^R(\bar z))\le E(\bar z_2,\partial_{z_2}\varphi(\bar z))$ and
 (\ref{eq:proof_convergence_sub_contradiction}) is equivalent to 
\begin{equation}
\label{eq:proof_convergence_sub_contradiction_bis}
\lambda \psi(\bar z_2)+E(\bar z_2,\psi'(\bar z_2))=\theta>0.
\end{equation}
Let  $\chi(\bar z_2,\psi'(\bar z_2),\cdot)$ be a solution of \eqref{cellpE} such that $\chi(\bar z_2,\psi'(\bar z_2),(0,0))=0$, see Theorem \ref{lem:existence_cellpE},
 and $W(\bar z_2,\psi'(\bar z_2),z_1)=\lim_{\epsilon\to 0}\epsilon\chi(\bar z_2, \psi'(\bar z_2),\frac z {\epsilon})$.

\paragraph{Step 1}
 We claim that for $\epsilon>0$ and $r>0$ small enough, the function $\varphi^\epsilon$:
 \[\varphi^\epsilon(z)=\psi(z_2)+\epsilon\chi(\bar z_2, \psi'(\bar z_2),\frac z {\epsilon})\] is a viscosity supersolution of 
\begin{equation}
\label{cellpE_modifed}
\left\{
    \begin{array}[c]{rcll}
   \lambda \varphi^\epsilon(z)+ \tilde{H}_\epsilon^i(z,D\varphi^\epsilon(z))& \ge& \frac{\theta}{2}\quad &\hbox{ if } z\in\Omega^i\cap B((0,\bar z_2),r),  \\
\lambda \varphi^\epsilon(z)+\tilde{H}_{\Ga,\epsilon} (z,D\left( \varphi^\epsilon\right)^L(z),D\left( \varphi^\epsilon\right)^R(z)) &\ge& \frac{\theta}{2}&\hbox{ if } z\in\Ga\cap B((0,\bar z_2),r),
    \end{array}
\right.
\end{equation}
where the Hamiltonians $\tilde{H}_\epsilon^i$ and $\tilde{H}_{\Ga,\epsilon} $ are defined by \eqref{eq:7new} and \eqref{eq:8new}. 
\\
Indeed, if $\xi$ is a test-function in $\cR$ such that $\varphi^\epsilon-\xi$ has a local minimum at  $z^\star\in B((0,\bar z_2),r)$, then, from the definition of $\varphi^\epsilon$,
 $y\mapsto \chi(\bar z_2, \psi'(\bar z_2),y)-\frac{1}{\epsilon}\left(\xi(\epsilon y)-\psi(\epsilon y_2) \right)$ has a local minimum at $\frac{z^\star}{\epsilon}$.
 Let us now use the fact that $y\mapsto \chi(\bar z_2, \psi'(\bar z_2),y)$ is a supersolution of \eqref{cellpE}:

\medskip 

If $\frac{z^\star}{\epsilon}\in \Omega^i$, for $i=L$ or  $R$, then 
$\tilde{H}^i((0,\bar z_2),D\xi(z^\star)-\psi'(z^\star_2)e_2+\psi'(\bar z_2)e_2,\frac{z^\star_2}{\epsilon})\ge E(\bar z_2,\psi'(\bar z_2))$.
From the regularity properties of   $\tilde{H}^i$, see Lemma \ref{sec:fast-slow-variables-2},
\begin{displaymath}
\tilde{H}^i((0,\bar z_2),D\xi(z^\star)-\psi'(z^\star_2)e_2+\psi'(\bar z_2)e_2,\frac{z^\star_2}{\epsilon})=\tilde{H}_\epsilon^i(z^\star,D\xi(z^\star)) + o_{\epsilon\to 0}(1)+o_{r\to 0}(1),    
\end{displaymath}
thus 
\begin{equation*}
\label{eq:proof_convergence_case1_nb1}
\begin{split}
&\lambda \varphi^\epsilon(z^\star)+ \tilde{H}_\epsilon^i(z^\star,D\xi(z^\star)) \\ \ge &E(\bar z_2,\psi'(\bar z_2))+\lambda\left(\psi(z^\star_2)+\epsilon\chi\left(\bar z_2, \psi'(\bar z_2),\frac{z^\star}{\epsilon}\right) \right)+o_{\epsilon\to 0}(1)+o_{r\to 0}(1).  
\end{split}
\end{equation*}
From \eqref{eq:proof_convergence_sub_contradiction_bis}, this implies that 
\begin{equation*}
\label{eq:proof_convergence_case1_nb2}
\lambda \varphi^\epsilon(z^\star)+ \tilde{H}_\epsilon^i(z^\star,D\xi(z^\star))   \ge  \theta+\lambda\epsilon\chi\left(\bar z_2, \psi'(\bar z_2),\frac{z^\star}{\epsilon}\right) 
+o_{\epsilon\to 0}(1)+o_{r\to 0}(1).
\end{equation*}
Recall that the function $y\mapsto \epsilon\chi(\bar z_2, \psi'(\bar z_2),\frac y {\epsilon})$ converges locally uniformly to  $y\mapsto W(\bar z_2, \psi'(\bar z_2),y)$,
which is a Lipschitz continuous function, independent of $y_2$ and such that   $W(\bar z_2, \psi'(\bar z_2),0)=0$.
 Therefore, for $ \epsilon$ and $r$ small enough, $\lambda \varphi^\epsilon(z^\star)+ \tilde{H}_\epsilon^i(z^\star,D\xi(z^\star))   \ge \frac{\theta}{2}$.
\medskip

If $\frac{ z^\star}{\epsilon}\in \Ga$, then, for some $i\in \{L,R\}$, $
 \tilde{H}^{+,i} (\bar z_2,D\xi^i(z^\star)-\psi'(z^\star_2)e_2+\psi'(\bar z_2)e_2,\frac{z^\star_2}{\epsilon})\ge E(\bar z_2,\psi'(\bar z_2))$.
Since the Hamiltonian $\tilde H^{+,i}$ enjoys the same regularity properties as  $\tilde H^i$, see Lemma~\ref{sec:fast-slow-variables-2},
 it is possible to use the same arguments as  in the case $\frac{z^\star}{\epsilon}\in \Omega^i$. For  $r$ and $\epsilon$ small enough,  
$\lambda \varphi^\epsilon(z^\star)+  \tilde{H}_{ \Ga,\epsilon}^{+,i}(z^\star,D\xi(z^\star))  \ge \frac{\theta}{2}$.
The claim that $\varphi^\epsilon$  is a supersolution of \eqref{cellpE_modifed} is proved.
\paragraph{Step 2} Let us prove that there exist some positive constants $K_r>0$ and  $\epsilon_0>0$ such that
\begin{equation}
\label{eq:proof_convergence_case1_nb3}
\tilde v^\epsilon(z)+K_r\le \varphi^\epsilon(z), \quad \forall z\in \partial B(\bar z, r), \;\forall \epsilon\in (0, \epsilon_0).
\end{equation}
Indeed, since $\overline{\tilde{v}}-\varphi$ has a strict local maximum at $\bar z$ and since $\overline{\tilde{v}}(\bar z)=\varphi(\bar z)$, 
there exists a positive constant $\tilde K_r>0$ such that 
$\overline{\tilde{v}}(z)+\tilde K_r\le \varphi(z)$ for any $z\in \partial B(\bar z, r)$.
Since $\ds \overline{\tilde{v}}={\limsup_\epsilon}^{*} \tilde{v}_\epsilon$,  there exists $\tilde \epsilon_0>0$ such that 
$\tilde v^\epsilon(z)+\frac{\tilde K_r}{2}\le \varphi(z)$
for any $0<\epsilon<\tilde \epsilon_0$ and $z\in \partial B(\bar z, r)$.
But $z\mapsto \varphi^\epsilon(z)$ converges locally uniformly to $z\mapsto \psi(z_2)+W(\bar z_2,\psi'(\bar z_2),z)$ as $\epsilon$ tends to $0$.
 Hence, thanks to  \eqref{eq:control_slopes_W_summary} in Corollary \ref{cor:control_slopes_W}, 
 \begin{displaymath}
   \psi(z_2)+W(\bar z_2,\psi'(\bar z_2),z)\ge
 \psi(z_2)+   \left( \overline \Pi^R   (\bar z_2,   \psi'(\bar z_2) ) 1_{z\in \Omega^R}+\overline \Pi^L(\bar z_2,   \psi'(\bar z_2)) 1_{z\in \Omega^L} \right) z_ 1,
 \end{displaymath}
and we get \eqref{eq:proof_convergence_case1_nb3} 
for some constants $K_r>0$ and $\epsilon_0>0$.

\paragraph{Step 3}
From the previous steps and the local comparison principle in Theorem \ref{th: local comparison_ps}, we find that for $r$ and $\epsilon$ small enough, 
\begin{displaymath}
\tilde v^\epsilon(z)+K_r\le \varphi^\epsilon(z) \quad \quad \forall z \in B(\bar z, r).
\end{displaymath}
Taking $z=\bar z$ and letting $\epsilon\to 0$,  we obtain
\begin{displaymath}
\overline{\tilde{v}}(\bar z)+K_r\le \psi(\bar z_2)=\varphi(\bar z)=\overline{\tilde{v}}(\bar z),
\end{displaymath}
which cannot happen. The proof is completed.

\begin{remark}
For the proof of the supersolution inequality, the test-function $\varphi$ should be chosen  of the form 
\begin{displaymath}
\varphi(z_1,z_2)=\psi(z_2)+1_{\Omega^R}\widehat{\Pi}^R(\bar z_2, \psi'(\bar z_2))z_1+1_{\Omega^L}\widehat{\Pi}^L(\bar z_2, \psi'(\bar z_2))z_1,
\end{displaymath}
where $\psi\in \cC^1(\R)$ and for $i=L,R$, $\widehat{\Pi}^i(\bar z_2, \psi'(\bar z_2))$ are defined in \eqref{eq:6} and \eqref{eq:10}.
\end{remark}

\appendix

\section{Proof of Lemma \ref{lem:stability_for_truncated_cell_problem}}
\label{appendix:proof_stability_truncated_cell_pb}
\paragraph{Subsolutions}
Let $\varphi\in \cR_\rho$ be a test-function and $\bar y\in [-\rho,\rho]\times \R$ be such that $u^0-\varphi$ has a strict local maximum at $\bar y$.
  If $\bar y\in (-\rho,0)\times \R$,  (resp. $\bar y \in (0,\rho)\times \R$ ) is  standard to check that   $\tilde H^L((0,z_2), D(\bar y)+p_2 e_2,\bar y_2)\le \lambda$,
 (resp. $\tilde H^R((0,z_2), D(\bar y)+p_2 e_2,\bar y_2)\le \lambda$).
  We may focus on the case when $\bar y=(0,\bar y_2)\in \Ga$, because the cases $\bar y_1 =\pm \rho$ can be treated with similar but simpler arguments.  
We wish to prove that
\begin{equation}
\label{eq:appendix_proof8stability_truncated1}
 \tilde{H}^{+,i} ((0,z_2),D\varphi^i(\bar y)+p_2e_2 \bar y_2)\le \lambda,   \quad \forall i=L,R.
\end{equation}
We may assume that for all $\eta\in [0,1]$, the function $u^\eta-\varphi$ is Lipschitz continuous in $[-\rho,\rho]\times \R$ 
with a Lipschitz constant $\bar L$ independent of $\eta$. Fix $i=L,R$, we define the distance $d_i$ to $\bar \Omega_i$ by
\begin{displaymath}
d_i(y)=\left\lbrace
\begin{array}{ll}
0 & \hbox{ if } y\in \Omega^i,\\
|y_1| & \hbox{ otherwise.}
\end{array}
\right.
\end{displaymath}
Clearly, $d_i\in \cR_\rho$. Take $C=\bar L+1$. The function $y\mapsto u^0(y)-\varphi(y)-Cd_i(y)$ has a strict local maximum at $\bar y$.
Thanks to the local uniform convergence of $u^\eta$ to $u^0$, there exists $r\in (0,\rho)$ and a sequence of points $y^\eta \in B(\bar y,r)$ such that
\begin{displaymath}
u^\eta(y)-\varphi(y)-Cd_i(y)\le u^\eta(y^\eta)-\varphi(y^\eta)-Cd_i(y^\eta), \quad \hbox{ for all } y\in B(\bar y,r).
\end{displaymath}
Up to the extraction of  a subsequence, we may assume that $y^\eta\to \bar y$ as $\eta$ tends to $0$.
Note that $y^\eta\in \bar \Omega^i$. Indeed, if it was not the case, then calling $\overline{y^\eta}=(0,y_2^\eta)\in  B(\bar y,r)$,
\begin{displaymath}
u^\eta(y^\eta)-\varphi(y^\eta)-\left( u^\eta(\overline{y^\eta})-\varphi(\overline{y^\eta})\right) \le \bar L |y_1^\eta|=\bar L d_i(y^\eta),
\end{displaymath}
and $
u^\eta(y^\eta)-\varphi(y^\eta)-Cd_i(y^\eta)\le u^\eta(\overline{y^\eta})-\varphi(\overline{y^\eta})-d_i(y^\eta)< u^\eta(\overline{y^\eta})-\varphi(\overline{y^\eta})-Cd_i(\overline{y^\eta})
$, in contradiction with the definition of $y^\eta$.\\
Up to the  extraction of subsequences, we can make out two cases:\\
{\bf Case 1:} $y^\eta\in \Ga$. We obtain $
\eta u^\eta(y^\eta)+\max_{i=L,R}\left(  \tilde{H}^{+,i} ((0,z_2),Du^i(y^\eta)+p_2e_2,y^\eta_2)\right) \le \lambda_\eta$,
and then \eqref{eq:appendix_proof8stability_truncated1} by letting $\eta\to 0$.
\\{\bf Case 2:} $y^\eta\in \Omega^i$. We obtain that
$
\eta u^\eta(y^\eta)+\tilde H^i((0,z_2),D\varphi(y^\eta)+p_2e_2,y^\eta_2)\le \lambda_\eta $,
and then by letting $\eta\to 0$ that $\tilde H^i((0,z_2),D\varphi(\bar y)+p_2e_2,\bar y_2)\le \lambda$,
from the continuity of  the Hamiltonian $\tilde H^i$. Finally, since $\tilde H^i-\tilde H^{+,i}\ge 0$, which yields 
that $ \tilde{H}^{+,i} ((0,z_2),D\varphi^i(\bar y)+p_2e_2 \bar y_2)\le \lambda$. \\
Since the arguments above can be applied for $i=L$ and $i=R$, we have obtained \eqref{eq:appendix_proof8stability_truncated1}.

\paragraph{Supersolutions}
Let $\varphi\in \cR_\rho$ be a test-function and $\bar y\in [-\rho,\rho]\times \R$ be such that $u^0-\varphi$ has a strict local minimum at $\bar y$.
 As above, we may focus  on the case when $\bar y\in \Ga$. We wish to prove that
\begin{equation}
\label{eq:appendix_proof8stability_truncated2bis}
 \max_{i=L,R}\left(\tilde{H}^{+,i} ((0,z_2),D\varphi^i(\bar y)+p_2e_2,\bar y_2)\right)\ge \lambda.
\end{equation}
Define
\begin{eqnarray*}
&\tilde p_0^L&=
\\ & \min&\left\{
  \begin{array}[c]{l}
p\in \R:  \\\tilde H^L((0,z_2),(\partial_{y_2}\varphi(\bar y)+p_2)e_2+pe_1,\bar y_2)=\tilde H^{+,L}((0,z_2),(\partial_{y_2}\varphi(\bar y)+p_2)e_2+pe_1,\bar y_2)    
  \end{array}
\right\},\\  
\\
&\tilde p_0^R&=\\
&\max&\left\{
  \begin{array}[c]{l}
p\in \R: \\\tilde H^R((0,z_2),(\partial_{y_2}\varphi(\bar y)+p_2)e_2+pe_1,\bar y_2)=\tilde H^{+,R}((0,z_2),(\partial_{y_2}\varphi(\bar y)+p_2)e_2+pe_1,\bar y_2)
  \end{array}\right\}
\end{eqnarray*}
Recall that thanks to Lemma \ref{sec:effect-hamilt-gamma-1} and Remark \ref{sec:fast-slow-variables-1},
\begin{eqnarray*}
\tilde H^{+,L}((0,z_2),(\partial_{y_2}\varphi(\bar y)+p_2)e_2+pe_1,\bar y_2)=\left\lbrace
\begin{array}{ll}
\tilde E_0^L(z_2,\partial_{y_2}\varphi(\bar y)+p_2,\bar y_2) & \hbox{ if } p\le \tilde p_0^L,\\
\tilde H^{1}((0,z_2),(\partial_{y_2}\varphi(\bar y)+p_2)e_2+pe_1,\bar y_2) & \hbox{ if } p\ge \tilde p_0^L,
\end{array}
\right.  
\\
\tilde H^{+,R}((0,z_2),(\partial_{y_2}\varphi(\bar y)+p_2)e_2+pe_1,\bar y_2)=\left\lbrace
\begin{array}{ll}
\tilde H^{2}((0,z_2),(\partial_{y_2}\varphi(\bar y)+p_2)e_2+pe_1,\bar y_2) & \hbox{ if } p\le \tilde p_0^R,\\
\tilde E_0^R(z_2,\partial_{y_2}\varphi(\bar y)+p_2,\bar y_2) & \hbox{ if } p\ge \tilde p_0^R.
\end{array}
\right.
\end{eqnarray*}
where $\tilde E_0^i(z_2,p_2,y_2)$ is defined in Remark \ref{sec:fast-slow-variables-1}.\\ 
We make out two cases:\\
{\bf Case 1:}  $\partial_{y_1}\varphi^L(\bar y)\ge \tilde p_0^L$, $\partial_{y_1}\varphi^R(\bar y)\le \tilde p_0^R$ and 
for each $i\in \{L,R\}$ 
\begin{equation}
\label{eq:appendix_proof8stability_truncated3}
\tilde H^{i}((0,z_2),D\varphi^i(\bar y)+p_2e_2,\bar y_2)=\max_{j=L,R}\left( \tilde H^{+,j}((0,z_2),D\varphi^j(\bar y)+p_2e_2,\bar y_2)\right).
\end{equation}
In this case, we can use a standard stability argument: there exists a sequence $y^\eta$ of local minimum points of $u^\eta-\varphi$ which converges to $\bar y$ as $\eta$ tends to $0$. If, for a subsequence still called  $y^\eta$, $y^\eta\in \Ga$, then 
$
\eta u^\eta (y^\eta)+\max_{i=L,R}\left( \tilde H^{+,i}((0,z_2),D\varphi^i(y^\eta)+p_2e_2, y^\eta_2)\right)\ge \lambda_\eta$,
because $u^\eta$ is a supersolution of \eqref{trunc-cellpapprox}, and \eqref{eq:appendix_proof8stability_truncated2bis} is obtained  by letting $\eta\to 0$. 
\\
If for a subsequence, $y^\eta\in \Omega^i$ for some $i\in \{L,R\}$, 
then $\eta u^\eta (y^\eta)+\tilde H^{i}((0,z_2),D\varphi(y^\eta)+p_2e_2, y^\eta_2)\ge \lambda_\eta$,
and by letting $\eta\to 0$, $\tilde H^{i}((0,z_2),D\varphi^i(\bar y)+p_2e_2, \bar y_2)\ge \lambda$.
Finally, \eqref{eq:appendix_proof8stability_truncated2bis} is obtained thanks to \eqref{eq:appendix_proof8stability_truncated3}.\\
{\bf Case 2:}  the assumptions of the case 1 are not satisfied. Thanks to 
the above identities for $\tilde H^{+,R}((0,z_2),(\partial_{y_2}\varphi(\bar y)+p_2)e_2+pe_1,\bar y_2)$ and $\tilde H^{+,L}((0,z_2),(\partial_{y_2}\varphi(\bar y)+p_2)e_2+pe_1,\bar y_2)$, 
by  modifying the slopes of $\varphi$ in the normal direction on each side of $\Ga$, it is possible to construct a test-function $\psi\in \cR_\rho$ such that
$\psi(\bar y)=\varphi(\bar y)$, $\partial_{y_2}\psi(\bar y)=\partial_{y_2}\varphi(\bar y)$, $\partial_{y_1}\psi^L(\bar y)\ge \partial_{y_1}\varphi^L(\bar y)$,
$\partial_{y_1}\psi^R(\bar y)\le \partial_{y_1}\varphi^R(\bar y)$, and  $\psi$ satisfies \eqref{eq:appendix_proof8stability_truncated3} for each $i\in \{L,R\}$.
Thus, since $\psi$ touches $\varphi$ at $\bar y$ from below,
 $\bar y$ is still a strict local minimum point of $u^0-\psi$ and we conclude by applying the result proved in the first case.

\section{Proof of Theorem \ref{th:restriction_set_of_test_functions}}
\label{appendix:proof_th_restriction_set_of_test_functions}
The proof of Theorem \ref{th:restriction_set_of_test_functions} follows the lines of \cite[Theorem~2.6]{imbert:hal-00832545} and relies on the following two technical lemmas,
which can be proved by adapting the proofs of Lemma 2.8 and Lemma 2.9 in \cite{imbert:hal-00832545}, with very slight changes. 
\begin{lemma}
\label{tech_lem:critcal_slope_on_each_branch_subsolution}
  Let $u:\R^2\to \R$ be a subsolution of \eqref{def:HJeffective1} and $\phi\in \cR$  touching $u$ from above at $\bar z=(0,\bar z_2)\in \Ga$. For each $i\in \{L,R\}$,  the real number $\bar p_i$:
  \begin{equation*}
  \bar p_i = \inf\left\lbrace p\in \R\; :\; \exists r>0   \hbox{ s.t. }     \phi(z)+\sigma^ipz_1\ge u(z), \; \forall z=(z_1,z_2)\in [0,r)\times (\bar z_2-r,\bar z_2+ r)
  \right\rbrace,
  \end{equation*}
  where $\sigma^i$ is given by \eqref{eq:def_sigma^i}, is nonpositive. Moreover,
  \begin{equation}
  \label{eq:lemma_critcal_slope_on_each_branch_subsolution}
\lambda u(0,\bar z_2) + H^i((0,\bar z_2),D \phi^i(0,\bar z_2)+\sigma^i\bar p_i e_1)\le 0.  
  \end{equation}
\end{lemma}
\begin{lemma}
\label{tech_lem:critcal_slope_on_each_branch_supersolution}
  Let $w:\R^2\to \R$ be a supersolution of \eqref{def:HJeffective1} and $\phi\in \cR$  touching $w$ from below 
 at $\bar z=(0,\bar z_2)\in \Ga$. For each $i\in \{L,R\}$. For each $i\in \{L,R\}$,  the real number $\tilde p_i $:
  \begin{equation*}
\tilde p_i = \sup\left\lbrace p\in \R \; :\; \exists r>0   \hbox{ s.t. } \phi(z)+\sigma^i pz_1\le w(z), \; \forall z=(z_1,z_2)\in [0,r)\times (\bar z_2-r,\bar z_2+ r)
  \right\rbrace,
  \end{equation*}
 is nonnegative.  Moreover
  \begin{equation}
  \label{eq:lemma_critcal_slope_on_each_branch_supersolution}
\lambda w(0,\bar z_2) + H^i((0,\bar z_2),D \phi^i(0,\bar z_2)+\sigma^i \tilde p_i e_1)\ge 0.  
  \end{equation}
\end{lemma}
We are now ready to prove Theorem \ref{th:restriction_set_of_test_functions}.  
\paragraph{Subsolutions}
Let $\Pi$ be a map as in Definition~\ref{def:test_functions_set_restricted}.  
 Suppose that a subsolution $u$  of (\ref{def:HJeffective1}) satisfies (\ref{eq:th_restriction_set_of_test_functions})  for all $z\in \Gamma$ and 
all test-functions in $\cR^\Pi(z)$ touching $u$ from above at  $z$.  
\\
 Let $\phi\in \cR$ be such that $u-\phi$ has a strict local maximum at $\bar z\in \Ga$ and that $u(\bar z)= \phi(\bar z)$.  We wish to prove that 
\begin{equation}
\label{eq:appendixA_proof1}
\lambda u(\bar z)+\max\left(E(\bar z_2,\partial_{z_2}\phi(\bar z)), H_\Ga(z, D \phi^L(\bar z),D \phi^R(\bar z)) \right) \le 0.
\end{equation}
From Lemma \ref{tech_lem:critcal_slope_on_each_branch_subsolution}, for each $i\in \{L,R\}$, there exists $\bar p_i\le 0$ such that
\begin{equation}
\label{eq:appendixA_proof2}
\lambda u(\bar z) + H^i(\bar z,D \phi^i(\bar z)+\sigma^i\bar p_i e_1)\le 0.  
\end{equation}
From the monotonicity properties of the Hamiltonians $H^{+,i}$  stated in  Lemma~\ref{sec:effect-hamilt-gamma-1},
\[
\begin{split}
H_\Ga(z, D \phi^L(\bar z),D \phi^R(\bar z)) &\le  H_\Ga(z, D \phi^L(\bar z)-\bar p_L e_1,D \phi^R(\bar z)+\bar p_R e_1)\\  
&\le \max \left( H^L(z, D \phi^L(\bar z)-\bar p_Le_1),H^R(z, D \phi^R(\bar z)+\bar p_Re_1)\right).
\end{split}
\]
Hence,  from \eqref{eq:appendixA_proof2},
\begin{equation*}
\lambda u(\bar z) +H_\Ga(\bar z, D \phi^L(\bar z),D \phi^R(\bar z)) \le 0.
\end{equation*}
Therefore, in order to prove \eqref{eq:appendixA_proof1}, we are left with checking that
\begin{equation}
\label{eq:appendixA_proof1bis}
\lambda u(\bar z)+E(\bar z_2,\partial_{z_2}\phi(\bar z)) \le 0.
\end{equation}
Recall that from Proposition \ref{cor:E_bigger_than_E_0}, $E(\cdot,\cdot)\ge E_0(\cdot,\cdot)$. If $E(\bar z_2,\partial_{z_2}\phi(\bar z))=E_0(\bar z_2,\partial_{ z_2}\phi(\bar z))$, then \eqref{eq:appendixA_proof1bis} is a direct consequence of \eqref{eq:appendixA_proof2}.
 Let us consider the case when $E(\bar z_2,\partial_{ z_2}\phi(\bar z))>E_0(\bar z_2,\partial_{ z_2}\phi(\bar z))$ and 
 assume by contradiction that
 \begin{equation}
   \label{eq:21}
\lambda u(\bar z)+E(\bar z_2,\partial_{ z_2}\phi(\bar z)) > 0. 
\end{equation}
Then, from \eqref{eq:appendixA_proof2} , for any $i\in \{L,R\}$,
\begin{equation*}
 H^{-,i}(\bar z,D \phi^i(\bar z)+\sigma^i\bar p_i e_1) \le-\lambda u(\bar z)<E(\bar z_2,\partial_{ z_2}\phi(\bar z)).
\end{equation*}
From this and the monotonicity properties of the functions $p\in \R\mapsto H^{-,i}(z,\partial_{z_2}\phi(\bar z) e_2+p e_1)$, we deduce that
\begin{displaymath}
\Pi^L(\bar z_2, \partial_{z_2}\phi(\bar z))  <  \partial_{z_1}\phi^L(\bar z)-\bar p_L,
\quad \hbox{and}\quad 
\Pi^R(\bar z_2, \partial_{z_2}\phi(\bar z))  >  \partial_{z_1}\phi^R(\bar z)+\bar p_R.
\end{displaymath}
Thus, the modified test-function $\varphi \in \cR^\Pi(\bar z)$ defined by
\begin{displaymath}
\varphi(z)=\phi(0,z_2)+1_{\Omega^L}(z)\Pi^L(\bar z_2, \partial_{z_2}\phi(\bar z))z_1+1_{\Omega^R}(z)\Pi^R(\bar z_2, \partial_{z_2}\phi(\bar z))z_1
\end{displaymath}
is such that  $u-\varphi$ has a local maximum at $\bar z$,  and therefore
\begin{equation*}
  \lambda u(\bar z)+\max\left(E(\bar z_2,\partial_{z_2}\varphi(\bar z)), H_\Ga(\bar z, D \varphi^L(\bar z),D \varphi^R(\bar z)) \right) \le 0,
\end{equation*}
which contradicts   (\ref{eq:21}).
\paragraph{Supersolutions}
 Suppose that a supersolution $u$  of (\ref{def:HJeffective1}) satisfies (\ref{eq:th_restriction_set_of_test_functions})  for all $z\in \Gamma$ and 
all test-functions in $\cR^\Pi(z)$ touching $u$ from below at  $z$.  
\\
Let $\phi\in \cR$ be such that $u-\phi$ has a strict local maximum at $\bar z\in \Ga$ with $u(\bar z)=\phi(\bar z)$. We wish to prove that 
\begin{equation}
\label{eq:appendixA_proof1_super}
\lambda u(\bar z)+\max\left(E(\bar z_2,\partial_{ z_2}\phi(\bar z)), H_\Ga(z, D \phi^L(\bar z),D \phi^R(\bar z)) \right) \ge 0.
\end{equation}
From Lemma \ref{tech_lem:critcal_slope_on_each_branch_supersolution}, for each $i\in \{L,R\}$ there exists $\tilde p_i\ge 0$ such that
\begin{equation}
\label{eq:appendixA_proof2_super}
\lambda u(\bar z) + H^i(\bar z,D \phi^i(\bar z)+\sigma^i\tilde p_i e_1)\ge 0,
\end{equation}
and using the monotonicity properties of the Hamiltonians $H^{+,i}$, see Lemma~\ref{sec:effect-hamilt-gamma-1},
\begin{equation}
\label{eq:appendixA_proof2_super_bis}
H_\Ga(\bar z, D \phi^L(\bar z),D \phi^R(\bar z)) \ge  H_\Ga(\bar z, D \phi^L(\bar z)-\tilde p_Le_1,D \phi^R(\bar z)+\tilde p_Re_1).
\end{equation}
If for some $i\in \{L,R\}$, $H^{+,i}(\bar z,D \phi^i(\bar z)+\sigma^i\tilde p_i e_1)=H^i(\bar z,D \phi^i(\bar z)+\sigma^i\tilde p_i e_1)$, 
then \eqref{eq:appendixA_proof1_super} follows readily from \eqref{eq:appendixA_proof2_super} and \eqref{eq:appendixA_proof2_super_bis}
so we may now suppose that
\begin{equation}
\label{eq:proof_reduction_set_of_test_function_key_inequality_1}
H^i(\bar z,D \phi^i(\bar z)+\sigma^i\tilde p_i e_1)=H^{-,i}(\bar z,D \phi^i(\bar z)+\sigma^i\tilde p_i e_1), \quad i=L,R.
\end{equation}
 Assume by contradiction that 
$\lambda u(\bar z)+\max\left(E(\bar z_2,\partial_{z_2}\phi(\bar z)), H_\Ga(\bar z, D \phi^L(\bar z),D \phi^R(\bar z)) \right)< 0$.
Thus, from \eqref{eq:appendixA_proof2_super} and (\ref{eq:proof_reduction_set_of_test_function_key_inequality_1}),
\begin{equation}
\label{eq:proof_reduction_set_of_test_function_key_inequality_2}
E(\bar z_2,\partial_{ z_2}\phi(\bar z)) < -\lambda u(\bar z) \le   H^{-,i}(\bar z,D \phi^i(\bar z)+\sigma^i\tilde p_i e_1), \quad i=L,R.
\end{equation}
From \eqref{eq:proof_reduction_set_of_test_function_key_inequality_1}, \eqref{eq:proof_reduction_set_of_test_function_key_inequality_2}
 and the monotonicity properties of the functions $p\in \R\mapsto H^{-,i}(z,\partial_{z_2}\phi(\bar z) e_2+p e_1)$, we deduce that
\begin{displaymath}
\Pi^L(\bar z_2, \partial_{z_2}\phi(\bar z))  >  \partial_{z_1}\phi^L(\bar z)-\tilde p_L,\quad \hbox{and}\quad 
\Pi^R(\bar z_2, \partial_{z_2}\phi(\bar z))  <  \partial_{z_1}\phi^R(\bar z)+\tilde p_R.
\end{displaymath}
Since the modified test-function $\varphi\in \cR^\Pi(\bar z)$,
\begin{displaymath}
\varphi(z)=\phi(0,z_2)+1_{\Omega^L}\Pi^L(\bar z_2, \partial_{z_2}\phi(\bar z))z_1+1_{\Omega^R}\Pi^R(\bar z_2, \partial_{z_2}\phi(\bar z))z_1,
\end{displaymath}
is such that  $u-\varphi$ has a local minimum at $\bar z$, we get
\begin{displaymath}
\lambda u(\bar z)+\max\left(E(\bar z_2,\partial_{z_2}\phi(\bar z)), H_\Ga(z, D \varphi^L(\bar z),D \varphi^R(\bar z))\right)\ge 0,
\end{displaymath}
which is the desired contradiction.

\paragraph{Acknowledgements}
The authors  were partially funded  by the ANR  project ANR-12-BS01-0008-01.
The first author  was partially funded  by the ANR project ANR-12-MONU-0013.
{\small
\bibliographystyle{amsplain}
\bibliography{homog_interface}
}
\end{document}